\documentclass[10pt,reqno,oneside]{amsart}
\usepackage[all]{xy}
\usepackage{amssymb,latexsym}
\usepackage{mltex}
\usepackage{amsmath}
\usepackage[colorlinks=true,linkcolor=blue,citecolor=blue]{hyperref}
 \usepackage{lscape}
\usepackage{enumerate}
\usepackage{amsthm}
\usepackage{hyperref}
\usepackage{multicol}

\usepackage[totalwidth=14cm,totalheight=22cm]{geometry}

\newtheorem{lem}{Lemma}[section]
\newtheorem{thm}[lem]{Theorem}
\newtheorem{cor}[lem]{Corollary}
\newtheorem{prop}[lem]{Proposition}

\theoremstyle{definition}
\newtheorem{defn}[lem]{Definition}

\theoremstyle{remark}
\newtheorem{rem}[lem]{Remark}

\newcommand{\R}{\mathbb{R}}

\newcommand{\Z}{\mathbb{Z}}
\newcommand{\N}{\mathbb{N}}

\DeclareRobustCommand{\gobblefive}[5]{}
\newcommand*{\SkipTocEntry}{\addtocontents{toc}{\gobblefive}}

\begin{document}
\title{Entire hypersurfaces of constant scalar curvature in Minkowski space}

\author{Pierre Bayard}

\address[P.~B.]{Facultad de Ciencias, Universidad Nacional Aut\'onoma de M\'exico
\\Av. Universidad 3000, Circuito Exterior S/N
\\Delegaci\'on Coyoac\'an, C.P. 04510, Ciudad Universitaria, CDMX, M\'exico}
\email{bayard@ciencias.unam.mx}

\author{Andrea Seppi}
\address[A.~S.]{Univ. Grenoble Alpes, CNRS, IF, 38000 Grenoble, France}
	\email{andrea.seppi@univ-grenoble-alpes.fr}

\maketitle

\begin{abstract}
We show that every regular domain $\mathcal D$ in Minkowski space $\mathbb R^{n,1}$ which is not a wedge admits an entire hypersurface whose domain of dependence is $\mathcal D$ and whose scalar curvature is a prescribed constant (or function, under suitable hypotheses) in $(-\infty,0)$. Under rather general assumptions, these hypersurfaces are unique and provide foliations of $\mathcal D$. As an application, we show that every maximal globally hyperbolic Cauchy compact flat spacetime admits a foliation by hypersurfaces of constant scalar curvature, generalizing to any dimension previous results of Barbot-B\'eguin-Zeghib (for $n=2$) and Smith (for $n=3$). 
\end{abstract}

%%%%%%%%%%%%%%%%%%%%%%%%%%%%%%%%%%%%%%%%%%%%%%%%%%%%%%%%%%%%%%%%%%%%%%%%%%%%%%%%%%%
%\noindent
%{\it Keywords: }.\\\\
%\noindent
%{\it 2020 Mathematics Subject Classification:} .

\date{}
%%%%%%%%%%%%%%%%%%%%%%%%%%%%%%%%%%%%%%%%%%%%%%%%%%%%%%%%%%%%%%%%%%%%%%%%%%%%%%%%%%%
\maketitle\pagenumbering{arabic}
%%%%%%%%%%%%%%%%%%%%%%%%%%%%%%%%%%%%%%%%%%%%%%%%%%%%%%%%%%%%%%%%%%%%%%%%%%%%%%%%%%%
%\tableofcontents

% !TEX root = main.tex

\tableofcontents

\section{Introduction}

The study of spacelike hypersurfaces in Minkowski space $\R^{n,1}$, and more generally in locally Minkowski manifolds, is a largely explored subject that connects differential geometry, geometric analysis, geometric topology, mathematical physics. A very natural class is that of \emph{entire} spacelike hypersurfaces, that is, which are graphs over $\R^n$, or equivalently which are properly embedded. 

While for \emph{maximal} hypersurfaces (i.e. of vanishing mean curvature) the Bernstein theorem holds in any dimension, meaning that every entire maximal hypersurface is a spacelike hyperplane --- as proved in the 1970s by Calabi \cite{calabi} for $n\leq 4$ and by Cheng and Yau \cite{chengyau} for any $n$ --- in the 1980s a large interest has grown on \emph{constant mean curvature} (CMC) hypersurfaces, which admit many non-trivial entire solutions. This interest was initially motivated by the relation with harmonic maps provided by the Gauss map, and the classification of entire CMC hypersurfaces has been completed in \cite{BSS2}, building on the earlier works \cite{T,CT}. More or less in the same period, in dimension $2+1$, Hano and Nomizu \cite{hanonomizu} have initiated the study of entire surfaces of  \emph{constant (negative) Gaussian curvature} --- that is, surfaces intrinsically locally isometric to the hyperbolic space, up to a factor.  Here progress has been made in \cite{Li95, schoenetal, BS}, and the full classification has been completed in \cite{BSS}.

There are (at least) two possible ways to generalize these results on constant Gaussian curvature surfaces in $\R^{2,1}$ to higher dimensions. One possibility is to consider hypersurfaces in $\R^{n,1}$ of \emph{constant Gauss-Kronecker curvature}, namely, such that the determinant of the shape operator is constant. Although there are in general serious regularity issues in higher dimensions, results in this direction have been obtained in \cite{Del, Li95, Guan, schoenetal, BaySchnu, fillver, bonfill,NS23CVPDE}.

The second possibility, which we develop in this work, is the classification problem for hypersurfaces of \emph{constant (negative) scalar curvature} in any dimension $n+1$. Equivalently, those hypersurfaces have a constant value of the second elementary symmetric polynomial of the principal curvatures. Let us also briefly mention that other symmetric functions of the principal curvatures were recently studied in \cite{WX2,WX,RWX,RWX2}.

\SkipTocEntry \subsection*{Asymptotic data}

Before stating our results, we need to take a step back and explain how the classification problem is formulated. In all the aforementioned classification results, the classifying data encode the asymptotic behaviour of the hypersurfaces. The information on such asymptotic behaviour can be expressed in several ways, which we now briefly describe. It is worth observing immediately that these formulations are all completely equivalent, at least for convex hypersurfaces --- and more generally for hypersurfaces which are trapped between two convex hypersurfaces with the same asymptotics, which will be the case here. 

Let us start by the ``geometric'' approach. Given an entire spacelike hypersurface $\Sigma$ in $\R^{n,1}$, its \emph{domain of dependence} is the set of points $p$ such that every inextensible timelike curve through $p$ intersects $\Sigma$. When $\Sigma$ has mean curvature bounded below by a positive constant, which will be the case, up to time reversal, for hypersurfaces of constant negative scalar curvature (see Subsection \ref{sec:scalar and admissible}), the domain of dependence is a \emph{future regular domain}, namely a convex domain obtained as a non-trivial intersection of future half-spaces bounded by lightlike hyperplanes. Any future regular domain can be written as the supergraph of the 1-Lipschitz convex function $V_\varphi:\R^n\to\R$ defined by
\begin{equation}\label{defVphi}
V_{\varphi}(x):=\sup_{y\in F_\varphi}\left( \langle x,y\rangle-\varphi(y)\right)~,
\end{equation}
where $\varphi: \mathbb{S}^{n-1} \rightarrow \R\cup\{+\infty\}$ is a
 lower semi-continuous function and $F_\varphi:=\{\varphi<+\infty\}$. We denote by $\mathcal D_\varphi$ the supergraph of $V_\varphi$. 
 
 The ``analytic'' approach, which is adopted among others in \cite{T,CT}, consists in expressing an entire spacelike hypersurface $\Sigma$ as the graph of a function $u:\R^n\to\R$ with $|Du|<1$, defining $L=\{y\in\mathbb S^{n-1}\,|\,\lim_{r\to+\infty}u(ry)/r=1\}$ and considering the function $f:L\to\R\cup\{+\infty\}$ defined by $f(y)=\lim_{r\to+\infty}(r-u(ry))$. When $\Sigma$ has mean curvature bounded below by a positive constant, $L$ coincides with the closure of $F_\varphi$ as above, and $\varphi|_L=f$, while  $\varphi\equiv+\infty$ on $\mathbb S^{n-1}\setminus L$. Hence finding a hypersurface of constant  scalar curvature $S<0$ with domain of dependence $\mathcal D_\varphi$ is completely equivalent to finding $u$ such that $L=\overline{F}_\varphi$, $f=\varphi|_L$ and the graph of $u$ has constant scalar curvature $S$. See \cite[Section 1.4]{BSS2} for the details of this equivalence and the additional viewpoint of the Penrose boundary, which will not be treated here.

\SkipTocEntry \subsection*{Geometric formulation}

The first result of this paper concerns existence and uniqueness of entire spacelike hypersurfaces with constant negative scalar curvature and prescribed domain of dependence $\mathcal{D}_{\varphi}$. We state it here in the ``geometric'' version.

\begin{thm}[Existence and uniqueness --- geometric version]\label{thm existence geometric}
Let $\mathcal D\subset \R^{n,1}$ be a future regular domain which is not a wedge. Then for every $S<0$, there exists an entire spacelike hypersurface $\Sigma$ of constant scalar curvature $S$ whose domain of dependence is $\mathcal D$. Moreover, if $\mathcal D=\mathcal D_\varphi$ for  $\varphi:\mathbb S^{n-1}\to\R$ a continuous function, then $\Sigma$ is unique.
\end{thm}

A \emph{wedge} is a regular domain obtained as the intersection of precisely two future half-spaces bounded by non-parallel lightlike hyperplanes. Equivalently, $F_\varphi=\{\varphi<+\infty\}$ consists of only two points. 

We also study foliations of $\mathcal D$. Recall that a \emph{time function} is a real-valued function on a Lorentzian manifold that is increasing along future-directed causal curves.

\begin{thm}[Foliation]\label{thm foliation}
Let $\varphi:\mathbb S^{n-1}\to\R$ be a continuous function. Then $\mathcal D_\varphi$ is foliated by hypersurfaces of constant scalar curvature in $(-\infty,0)$. Moreover,  the function associating to $p\in \mathcal D_\varphi$ the value of the scalar curvature of the unique leaf of the foliation containing $p$ is a {time function}.
\end{thm}

Actually, taking products, the conclusion holds for $\varphi:\mathbb S^{n-1}\to\R\cup\{+\infty\}$  any function which is continuous and real-valued on $\mathbb S^{n-1}\cap A$, where $A$ is an affine subspace of $\R^n$ of dimension $2\leq k\leq n$ intersecting $\mathbb S^{n-1}$ nontrivially, and is identically equal to $+\infty$ on the complement of $A$ --- that is, when the domain $\mathcal D$ is isometric to the product of a regular domain in $\R^{k,1}$ and of $\R^{n-k}$.

Theorems \ref{thm existence geometric} and \ref{thm foliation} were proved for $n=2$ (in that case, the scalar curvature is equal to twice the Gaussian curvature) in \cite{BSS}.

\SkipTocEntry \subsection*{Analytic formulation}

Let us now introduce the PDE set-up for Theorem \ref{thm existence geometric}. As already explained, up to time reversal, a hypersurface of constant and negative scalar curvature necessarily has positive mean curvature (Subsection \ref{sec:scalar and admissible}). It is thus natural to look for a solution of the problem in the space of spacelike $C^2$ functions $u:\R^n\to\R$ such that $\mathcal H_1[u]>0$ and $\mathcal H_2[u]>0$, where $\mathcal H_k[u]$ denotes the $k^{th}$ elementary symmetric polynomial of the principal curvatures of $u$, normalised so as to be equal to $1$ on the function $u$ whose graph is the future unit hyperboloid. In other words, $\mathcal H_1$ is a positive multiple of the mean curvature, and $\mathcal H_2$ is a negative multiple of the scalar curvature. We call these functions \emph{admissible}. Theorem \ref{thm existence} can be rewritten as follows (recall that $F_\varphi=\{\varphi<+\infty\}$).

\begin{thm}[Existence and uniqueness --- analytic version]\label{thm existence}
If $\varphi:\mathbb{S}^{n-1}\rightarrow\R\cup\{+\infty\}$ is a lower semi-continuous function such that $card(F_\varphi)\geq 3,$ then there exists an admissible function $u:\R^n\rightarrow\R$ such that $\mathcal{H}_2[u]\equiv 1$ and whose graph has domain of dependence $\mathcal{D}_{\varphi}.$ Moreover, if $\varphi$ is a continuous function then the admissible solution is unique.
\end{thm}

The second result deals with the additional requirement of the prescription of the scalar curvature as a function on the domain of dependence $\mathcal{D}_{\varphi}$:
\begin{thm}[Prescribed curvature function]\label{thm existence 2}
Let $\varphi:\mathbb{S}^{n-1}\rightarrow\R\cup\{+\infty\}$ be a lower semi-continuous function such that $card(F_\varphi)\geq 3$ and   $F_\varphi$ is not included in any affine hyperplane of $\R^n$ and let $H:\mathcal{D}_\varphi\subset\R^{n,1}\rightarrow\R$ be a function of class $C^{k,\alpha},$ $k\geq 2$, $\alpha\in (0,1),$ such that $h_0\leq H\leq h_1$ for some positive constants $h_0$ and $h_1$. Then there exists an admissible function $u:\R^n\rightarrow\R$ {belonging to $C^{k+2,\alpha}$} such that $\mathcal{H}_2[u]=H(\cdot,u(\cdot))$ on $\R^n$ and whose graph has domain of dependence $\mathcal{D}_{\varphi}.$ Moreover, if $\varphi$ is a continuous function and $\partial_{x_{n+1}}H\geq 0$, then the solution is unique.
\end{thm}

Clearly Theorem \ref{thm existence 2} could be also reformulated in geometric terms as we did for Theorem \ref{thm existence geometric}. We prefer to omit the precise statement, that can be easily deduced. We instead remark that some hypothesis on the function $H$ must be imposed (in the statement above, $\partial_{x_{n+1}}H\geq 0$) in order to ensure uniqueness. Indeed, if $H$ is the opposite of the time function given by Theorem \ref{thm foliation}, any hypersurface of the foliation is a solution of the problem $\mathcal{H}_2[u]=H(.,u)$.

\SkipTocEntry \subsection*{About the hypotheses on $\varphi$}

Theorems \ref{thm existence} and \ref{thm existence 2} are improvements of previous results in the literature. In \cite{Ba1}, the existence was proved under the hypothesis $\varphi\in C^2(\mathbb S^{n-1})$, while in \cite{Ba2} for any lower semi-continuous function that only takes the values $0$ or $+\infty$. When $n=2$, Theorem \ref{thm existence} is the main result of \cite{BSS}. 

The assumption that $\varphi$ is a lower semi-continuous function is not restrictive in any way. Indeed, any regular domain can be written as the supergraph of $V_\varphi$ as in \eqref{defVphi}, for $\varphi$ a lower semi-continuous function that takes finite values on at least two points. In order to achieve the sharpest possible existence result, it thus remains to address the question of whether the condition that $card(F_\varphi)\geq 3$ is a necessary condition.  For $n=2$, this was proved in \cite[Corollary C]{BSS2}. We answer this question affirmatively for $n=3$.

\begin{thm}\label{thm finitness2}
Let $\Sigma$ be an entire spacelike hypersurface in $\R^{3,1}$ with  scalar curvature bounded above by a negative constant. 
Then, up to a time-reversing isometry, $\mathcal D(\Sigma)=\mathcal D_\varphi$ where $card(F_\varphi)\geq 3$. 
\end{thm}

The strategy to prove Theorem \ref{thm finitness2} is the following. 
By contradiction, we suppose that $\Sigma$ is a hypersurface of scalar curvature bounded above by a negative constant, whose domain of dependence is a wedge $W$. Up to isometries, we suppose $W=\{x_{4}>|x_1|\}$.  As said before, the hypothesis implies that the mean curvature is bounded below by some $c>0$, up to time reversal (see \eqref{eq:silly inequality}). By the comparison principle for mean curvature proved in \cite[Theorem 2.1]{BSS2} (see Theorem \ref{thm:comparison CMC}), $\Sigma$ stays below the hypersurface of constant mean curvature $c$ in the wedge $W$, which is the product of a hyperbola in the timelike plane $P=\{x_2=x_3=0\}$ and of $P^\perp$. Hence the intersection of $\Sigma$ with the timelike hyperplane $\{x_1=0\}$, which is a copy of $\R^{2,1}$, is contained between two parallel spacelike planes, and we observe that it has positive mean curvature too. Then we show that there is no entire spacelike surface in $\R^{2,1}$ of positive mean curvature contained between two parallel spacelike planes, thus giving a contradiction.

Unfortunately, this argument does not extend to higher dimensions. In fact, using radially symmetric functions on $\R^n$, we show that for $n\geq 3$ one can find an entire spacelike hypersurface in $\R^{n,1}$ of positive mean curvature contained between two parallel spacelike hyperplanes. For $n\geq 5$, it can be constructed in such a way that the scalar curvature is also negative, and completing with hyperbolas provides a hypersurface of negative scalar curvature whose domain of dependence is the wedge in $\R^{n+1,1}$. (See Proposition \ref{prop higher dimensions radially}.) This shows that an extension of Theorem \ref{thm finitness2} in arbitrary dimension would require a substantially different strategy. 

Also, an entire hypersurface with scalar curvature bounded above by a negative constant and with domain of dependence a wedge, if it exists, cannot be convex (Remark \ref{rmk:nonconvexity}). We do not know if non-convex entire hypersurfaces of constant negative scalar curvature do exist. Let us briefly mention some results for $n=3$. A particular case of the main result of \cite{RWX2} shows that any entire hypersurface of constant scalar curvature and bounded principal curvatures is necessarily convex. As a consequence of \cite{grahamsmith}, strictly convexity of the constant scalar curvature hypersurfaces holds for those regular domains that admit a MGHC quotient (see the more detailed discussion below). Moreover, a Splitting Theorem (c.f. Theorem \ref{splitting theorem} for CMC hypersurfaces) holds to classify the non-strictly-convex solutions $\Sigma$ under suitable assumptions, see \cite[Theorem 2.7.2]{grahamsmith} and \cite{RWX2}.

\SkipTocEntry \subsection*{Strategy of proof of the existence}

Let us now outline the strategy to prove Theorems \ref{thm existence} and \ref{thm existence 2}. We use an exhaustion method, namely we solve a suitable sequence of finite Dirichlet problems over balls. Then we exploit the $C^1$ and $C^2$ estimates developed in \cite{Ba1} and \cite{Ur} respectively, to show that these solutions converge smoothly to an entire hypersurface with the right asymptotic behaviour.

The essential element to implement this strategy is the presence of (upper and lower) barriers --- meaning that any solution of the finite Dirichlet problem with Dirichlet condition bounded between the barriers still stays between the barriers --- having domain of dependence $\mathcal D$. The precise properties of the barriers are listed in Definition \ref{defi barriers}. We would like to underline only two important aspects here.

First, the construction of the lower barrier is an essential novel ingredient in this article, and is the key element that permits to largely  improve the previous existence results in the literature. Typically, barriers are taken to be smooth sub/super-solutions to the constant (or prescribed) scalar curvature problem. Here our lower barrier is constructed as the supremum of functions obtained from surfaces of constant Gaussian curvature in subspaces $Q$ of signature $(2,1)$, determined by triples of points in $F_\varphi$, and taking products with $Q^\perp$. In other words, the lower barrier is a sub-solution only in the viscosity sense. Our proof thus relies crucially on the existence of a surface of constant Gaussian curvature asymptotic to the cone in $\R^{2,1}$ obtained as the intersection of three pairwise non-parallel lightlike planes --- this is proved in \cite{BSS} and is used  there as an ingredient to prove the statement of Theorem \ref{thm existence geometric} for $n=2$.

Second, in order to apply Urbas' $C^2$ estimates, one needs a \emph{strictly convex} upper barrier. Concretely, we choose the upper barrier to be the CMC hypersurface whose regular domain is $\mathcal D$, whose existence is proved in \cite{BSS2}.  As a consequence of the Splitting Theorem (see Theorem \ref{splitting theorem} and Corollary \ref{cor:splitting}), this CMC hypersurface is strictly convex only if $\mathcal D$ does not split as the product of a regular domain in a copy of a lower-dimensional $\R^{k,1}$, and of $\R^{n-k}$, and this allows us to prove Theorems \ref{thm existence} and \ref{thm existence 2} in this case. The proof of Theorem \ref{thm existence} --- that is, for a regular domain $\mathcal D$ that might also split as a product --- then follows simply by taking products and an inductive argument. 

This discussion also explains why in Theorem \ref{thm existence 2} we need to assume that 
$F_\varphi$ is not included in some affine hyperplane of $\R^n$. If $F_\varphi$ is included in some affine subspace of $\R^n$ of minimal dimension $k$, we can assume that 
 $F_\varphi\subset\R^k\times\{0\}\subset\R^n,$ so that $\mathcal{D}_{\varphi}=\mathcal{D}_{\varphi'}\times\R^{n-k}$ where $\mathcal{D}_{\varphi'}$ is a regular domain in $\R^{k,1}.$ In this case, we can only solve the prescribed scalar curvature problem as in Theorem \ref{thm existence 2} for a function $H$ that does not depend on the second factor, i.e. is of the form $H(x',x'')=c_{k,n} H'(x')$ for  $H':\mathcal{D}_{\varphi'}\rightarrow\R$ and $c_{k,n}:=k(k-1)/n(n-1).$ Indeed, Theorem \ref{thm existence 2} yields an admissible solution of $\mathcal{H}_2[u']=H'$ in $\mathcal{D}_{\varphi'}$ and therefore an admissible solution $u(x',x'')=u'(x')$ of $\mathcal{H}_2[u]=H$ in $\mathcal{D}_{\varphi}$.

 \SkipTocEntry \subsection*{MGHC flat spacetimes}

An application of our existence, uniqueness and foliation results concern the so-called maximal globally hyperbolic Cauchy compact (MGHC) flat spacetimes. These have been first studied in \cite{Ba05} and \cite{Bo}, extending to arbitrary dimension the results of \cite{mess} in dimension $2+1$. In \cite{Ba05} a classification of MGHC flat spacetimes was provided, showing that (up to finite coverings) they are divided in three classes: translation spacetimes, Misner spacetimes and twisted products of Cauchy hyperbolic spacetimes (which have also been carefully studied in \cite{Bo}) with Euclidean tori. The latter class is by far the most interesting, and is obtained as the quotient of the product of a regular domain $\mathcal D_\varphi$, for $\varphi$ a continuous function on $\mathbb S^{n-d-1}$, and $\R^d$ for $d\geq 0$. Theorems \ref{thm existence geometric} and \ref{thm foliation} then imply the following result,  for $M$ of any dimension $n+1$.

\begin{thm}\label{thm:MGHC}
Let $M$ be a maximal globally hyperbolic Cauchy compact flat spacetime. Then $M$ has a foliation by closed hypersurfaces of constant scalar curvature. Unless $M$ is finitely covered by a translation spacetime or a Misner spacetime, every closed spacelike hypersurface in $M$ of constant scalar curvature coincides with a leaf of the foliation, the scalar curvature of the leaves varies in $(-\infty,0)$, and it defines a time function on $M$.
\end{thm}

When $M$ is finitely covered by a translation or Misner spacetime, the foliation by hypersurfaces of constant scalar curvature is very easily described --- actually, all leaves of the foliation have vanishing sectional curvature. (Uniqueness of the foliation, however, does not hold for these cases, see Remarks \ref{rmk:translation not unique} and \ref{rmk:Misner not unique}.)

The non-trivial part of Theorem \ref{thm:MGHC} hence concerns the case of twisted products of Cauchy hyperbolic spacetimes with Euclidean tori. Here we apply Theorems \ref{thm existence geometric} and \ref{thm foliation}. The uniqueness part of Theorem \ref{thm existence geometric} can be applied since $\varphi$ is continuous for the universal cover of Cauchy hyperbolic spacetimes and  is crucial to ensure that the hypersurfaces of constant scalar curvature induce closed hypersurfaces in the quotient.

Observe also that in dimension 2+1, foliations by constant Gaussian curvature of MGHC flat spacetimes where initially constructed in \cite{barbotzeghib}, see also \cite[Theorem D]{BSS}. Theorem \ref{thm:MGHC} was proved in \cite{grahamsmith} for MGHC flat spacetimes $M$ of dimension 3+1, by completely different methods. More precisely,  the approach of \cite{grahamsmith} consists in studying (in any dimension) the so-called special Lagrangian curvature, which in dimension 3+1 coincides with scalar curvature.

Finally, we remark that the existence of closed hypersurfaces of constant, or prescribed, scalar curvature in globally hyperbolic Cauchy compact Lorentzian manifolds was proved in \cite{Ger03} {under the assumptions} of existence of  barriers and of a strictly convex function between the barriers. As explained in the previous subsection, the existence of a (non-smooth) lower barrier is an essential ingredient of our proof of Theorem \ref{thm existence geometric}, and therefore of Theorem \ref{thm:MGHC}.
In other words, Theorem \ref{thm:MGHC} could not be proved by applying directly the results of \cite{Ger03}, in lack of a suitable lower barrier. 

\SkipTocEntry \subsection*{Acknowledgements}

The second author would like to thank Graham Smith for many ad hoc explanations and for pointing out relevant references, and to Thierry Barbot, Francesco Bonsante and Peter Smillie for several discussions   related to this work.

The second author is funded by the European Union (ERC, GENERATE, 101124349). Views and opinions expressed are however those of the author(s) only and do not necessarily reflect those of the European Union or the European Research Council Executive Agency. Neither the European Union nor the granting authority can be held responsible for them.

%\input{prel_analytic}

% !TEX root = main.tex

\section{Spacelike hypersurfaces and domains of dependence}

The $(n+1)$-dimensional Minkowski space $\R^{n,1}$ is the vector space $\R^{n+1}$ endowed with the flat Lorentzian metric
\begin{equation}\label{eq:metric}
g=dx_1+\cdots +dx_n^2-dx_{n+1}^2~.
\end{equation}

Its isometry group is the group $\mathrm O(n,1)\rtimes \R^{n+1}$. A vector $v$ is \emph{spacelike} (resp. \emph{lightlike}, \emph{timelike}) if $g(v,v)$ is positive (resp. null, negative). An immersed submanifold is spacelike if all its non-zero tangent vectors are spacelike; equivalently, its first fundamental form is a Riemannian metric. 

\subsection{Entire graphs}
The main object of this article is the study of spacelike hypersurfaces having certain curvature properties. We will restrict to the natural class of \emph{entire} spacelike hypersurfaces, which are defined by the equivalent conditions of the following proposition.

\begin{prop}[{\cite[Proposition 1.10]{BSS}}] \label{prop entire}
Let $\Sigma$ be an immersed smooth spacelike hypersurface of $\R^{n,1}$. Then the following are equivalent:
\begin{itemize}
\item[$i)$] $\Sigma$ is properly immersed; 
\item[$ii)$] $\Sigma$ is properly embedded; 
\item[$iii)$] 
$\Sigma$ is the graph of a function $u:\R^n\to\R$.
\end{itemize}
Moreover, if the conditions hold, then $u$ is a smooth function satisfying $|Du(x)|<1$ for all $x\in\R^n$. 
\end{prop}

We call a function $u$ as in the conclusion of Proposition \ref{prop entire} a \emph{spacelike function}. As a consequence, we remark that the condition of being the graph of a spacelike function is invariant under isometries of $\R^{n,1}$. We call  \emph{entire} the submanifolds $\Sigma$ satisying the equivalent conditions of Proposition \ref{prop entire}. 

\subsection{Domains of dependence and regular domains}

A \emph{causal curve} is a smooth curve $\gamma$ such that its tangent vector is either timelike or lightlike at any point. It is locally the graph of a function  $f:(a,b)\to\R^n$ satisfying $| Df(t) |\leq 1$ for every $t\in (a,b)$. The causal curve $\gamma$ is \emph{inextensible} if it is  the graph of a globally defined function $\R$ to $\R^n$, that is, $\gamma=\{(f(t),t)\,|\,t\in\R\}$ with  $ | Df(t) |\leq 1$ for every $t\in \R$. 

We can now define the domain of dependence. This could be defined more generally for achronal sets, but for the purpose of this work it is sufficient to restrict to entire spacelike hypersurfaces. 

\begin{defn}\label{defn:dod}
Let $\Sigma$ be a spacelike entire hypersurface. The domain of dependence $\mathcal D(\Sigma)$ of $\Sigma$ is the set of points $p$ of $\R^{n,1}$ such that every inextensible causal curve containing $p$ intersects $\Sigma$. 
\end{defn}

It can be shown that $\mathcal D(\Sigma)$ is open, and is the intersection of all open half-spaces $\mathcal H$ such that $\partial\mathcal H$ is a lightlike hyperplane and $\mathcal H$ contains $\Sigma$ --- see for example \cite[Section 3]{Bo} or \cite[Lemma 1.4]{BSS}.
 
We now introduce a special type of convex subsets of $\R^{n,1}$, that arise as domains of dependence of entire spacelike submanifolds. Given a lower semi-continuous function $\varphi: \mathbb{S}^{n-1} \rightarrow \R\cup\{+\infty\}$, we define the set 
$$F_\varphi:=\{y\in \mathbb{S}^{n-1}|\ \varphi(y)<+\infty\}$$ 
and the function $V_{\varphi}:\R^n\to\R$:

\begin{equation}
V_{\varphi}(x):=\sup_{y\in F_\varphi}\left( \langle x,y\rangle-\varphi(y)\right)~.
\end{equation}

Observe that, if $card(F_\varphi)=1$, that is if $\varphi$ is a function taking a finite value at a single point $y_0$ and $+\infty$ elsewhere, then $V_\varphi$ is the function $x\mapsto \langle x,y_0\rangle-\varphi(y_0)$, whose graph is a lightlike hyperplane. We call its open supergraph a \emph{future half-space}, and its open subgraph a \emph{past half-space}.

\begin{defn}\label{defn:regular domain}
A \emph{future regular domain} is a subset of $\R^{n,1}$ of the form:
\begin{equation}\label{eq:Dphi}
\mathcal{D}_{\varphi}:=\{(x,x_{n+1})\in\R^{n,1}|\ x_{n+1}> V_{\varphi}(x)\}~,
\end{equation}
for some lower semi-continuous function $\varphi: \mathbb{S}^{n-1} \rightarrow \R\cup\{+\infty\}$ such that $card(F_\varphi)\geq 2$.
\end{defn}
 
In the definition above, it is not restictive to suppose $\varphi$ lower semi-continuous, because, for any function $\phi$, the set $\mathcal{D}_{\phi}$ formally defined as in \eqref{eq:Dphi} is equal to $\mathcal{D}_{\varphi}$ where $\varphi$ is the lower semi-continuous envelope of $\phi$. One can define a past regular domain analogously, replacing lower semi-continuous functions by upper semi-continuous functions, changing minus by plus and replacing supremum by infimum in \eqref{defVphi},  and reversing the inequality in \eqref{eq:Dphi}. 

Now, observe that the intersection of a future half-space $\mathcal H_1$ and a past half-space $\mathcal H_2$ can contain no entire spacelike hypersurface unless $\partial \mathcal H_1$ and $\partial \mathcal H_2$ are parallel. In conclusion of the above discussion, the domain of dependence of an entire spacelike hypersurface can be: 
\begin{itemize}
\item the whole $\R^{n,1}$, 
\item a future half-space,
\item a past half-space,
\item the intersection of a future and a past half-space with parallel boundaries, 
\item a future regular domain, or 
\item a past regular domain. 
\end{itemize}

We will see in Corollary \ref{cor:scalar reg dom} that the domain of dependence of an entire spacelike hypersurface of scalar curvature bounded above by a negative constant is necessarily a (future or past) regular domain. 
Up to applying a time-reversing isometry, we will restrict ourselves to the case of future regular domains. In this setting, the function $\varphi$ determining the domain of dependence of $\Sigma$ is easily recovered via the following statement.

\begin{prop}\label{prop recover varphi}
Let $\Sigma=graph(u:\R^n\to \R)$ be an entire spacelike hypersurface whose domain of dependence $\mathcal D(\Sigma)$ is a future regular domain. Then $\mathcal D(\Sigma)=\mathcal D_\varphi$ where $\varphi: \mathbb{S}^{n-1} \rightarrow \R\cup\{+\infty\}$ is the following function: 
\begin{equation}\label{eq recover varphi}
\varphi(y)=\sup_{x\in\R^n}(\langle x,y\rangle-u(x))~.
\end{equation}
\end{prop}
\begin{proof}
We have already mentioned that $\mathcal D(\Sigma)$ is the intersection of all half-spaces containing $\Sigma$ whose boundary is a lightlike hyperplane. Moreover, since $\mathcal D(\Sigma)$ is a future regular domain, all such half-spaces are future. Now, for a given $y\in \mathbb S^{n-1}$, the future half-space whose boundary is the graph of $x\mapsto \langle x,y\rangle-c$ contains $\Sigma$ if and only if $c> \langle x,y\rangle-u(x)$ for all $x\in\R^n$. Hence $\mathcal D(\Sigma)$ is the intersection of the open future half-spaces whose boundaries are the lightlike hyperplanes $x\mapsto \langle x,y\rangle-\varphi(y)$, where $\varphi$ is as in \eqref{eq recover varphi}, for $\varphi(y)<+\infty$. Using Definition \ref{defn:regular domain}, this concludes the proof. 
\end{proof}

\subsection{Products}\label{subsec:products}

A simple way of producing (future) regular domains is to start from a (future) regular domain contained in a lower-dimensional copy of $\R^{k,1}$, and take a product with its orthogonal complement. Since those regular domains obtained in this way play an important role for Theorem \ref{thm existence}, we introduce a definition.

\begin{defn}\label{def split domain}
Let $\mathcal D$ be a regular domain in $\R^{n,1}$. We say that $\mathcal D$ \emph{splits} if there exists a subspace $P$ of signature $(k,1)$, for $1\leq k\leq n-1$, such that 
\begin{equation}\label{eq:form split domain}
\mathcal D=\{x+v\,|\,x\in\mathcal D',v\in  P^\perp\}
\end{equation}
where $\mathcal D'$ is a regular domain in $P\cong \R^{k,1}$.
\end{defn} 

For brevity, when $\mathcal D$ splits as above, we will write that $\mathcal D\cong \mathcal D'\times\R^{n-k}$. The next proposition characterizes future regular domains that split in terms of the function $\varphi$.

\begin{prop}\label{prop:char split domain}
Let $\mathcal D_\varphi$ be a future regular domain, for a lower semi-continuous function $\varphi: \mathbb{S}^{n-1} \rightarrow \R\cup\{+\infty\}$ such that $card(F_\varphi)\geq 2$. Then $\mathcal D_\varphi$ splits if and only if there exists an affine subspace $A$ of dimension $k\leq n-1$ of $\R^n$ such that $F_\varphi\subset \mathbb{S}^{n-1}\cap A$. 
\end{prop}
\begin{proof}
First, observe that $F_\varphi\subset \mathbb{S}^{n-1}\cap A$ if and only if all lightlike hyperplanes of the form 
\begin{equation}\label{eq:hyperplanes again}
x_{n+1}=\langle x,y\rangle-\varphi(y)
\end{equation}
for $y\in F_\varphi$, are invariant under translations in $P^\perp$, where $P:=\overline{\{\lambda(a,1)\,|\,a\in A,\lambda\in\R\}}\subset \R^{n,1}$. So if there exists $A$ such that $F_\varphi\subset \mathbb{S}^{n-1}\cap A$, then $\mathcal D_\varphi$, which is the intersection of the future half-spaces bounded by hyperplanes of the form \eqref{eq:hyperplanes again}, is invariant under translations in $P^\perp$, and hence a product of $\mathcal D':=\mathcal D_\varphi\cap P$ and $P^\perp$ as in \eqref{eq:form split domain}. 

Conversely, suppose that $\mathcal D_\varphi$ is of the form \eqref{eq:form split domain} for $\mathcal D'\subset P$. Then $\mathcal D_\varphi$ contains an affine subspace parallel to $P^\perp$. This implies that  if a future half-space bounded by a hyperplane of the form \eqref{eq:hyperplanes again} contains $\mathcal D_\varphi$, then $(y,1)\in P$, namely,  $y\in A:=P\cap \{x_{n+1}=1\}$. This shows that $F_\varphi\subset A$ and concludes the proof.
\end{proof}

\section{Scalar and mean curvature}

\subsection{Some general formulae}
Let $\Omega$ be a convex domain in $\R^n$ (most of the time, $\Omega=\R^n$), let $u:\Omega\to\R$ be a smooth spacelike function (i.e. $|Du(x)|<1$ for all $x\in\Omega$), and let $\Sigma$ be the graph of $u$. When $\Omega=\R^n$, $\Sigma$ is thus an entire spacelike hypersurface (see Proposition \ref{prop entire}). By an elementary computation, the first fundamental form of $\Sigma$ is
$$g_{ij}=\delta_{ij}-\partial_i u\ \partial_ju$$
and the second fundamental form, computed with respect to the \emph{future} unit normal vector
$$N=\frac{1}{\sqrt{1-|Du|^2}}(\partial_1u,\ldots,\partial_nu,1)~,$$
is
$$h_{ij}=\frac{1}{\sqrt{1-|Du|^2}}\ \partial^2_{ij}u.$$
Since the inverse of the metric is
$$g^{ij}=\delta_{ij}+\frac{\partial_iu\ \partial_j u}{1-|Du|^2}$$
the shape operator of $\Sigma$ is given by 
\begin{equation}\label{expr shape op}
h^i_j=\frac{1}{\sqrt{1-|Du|^2}}\sum_{k=1}^n\left(\delta_{ik}+\frac{\partial_i u\ \partial_k u}{1-|Du|^2}\right)\partial^2_{kj}u.
\end{equation}
Let us denote by {$\lambda_1,\ldots,\lambda_n$ the principal curvatures of $\Sigma$ (the eigenvalues of the shape operator) and by $\mathcal{H}_k[u]$ the normalised $k^{th}$ elementary symmetric function of $\lambda_1,\ldots,\lambda_n$ (the $k^{th}$-curvature)} 
$$\mathcal{H}_k[u]=\frac{k!(n-k)!}{n!}\ \sigma_k(\lambda_1,\ldots,\lambda_n).$$ 

We are interested in the scalar curvature $S[u]$ of $\Sigma$ which is linked to $\mathcal{H}_2[u]$ by the identity:
$$S[u]=-n(n-1)\ \mathcal{H}_2[u]~,$$
and in the mean curvature, which is simply equal to $\mathcal{H}_1[u]$, and is given by:
$$\mathcal{H}_1[u]=\frac{1}{n}\sum_{i=1}^n h_i^i=\frac{1}{n\sqrt{1-|Du|^2}}\sum_{1\leq i,j\leq n}\left(\delta_{ij}+\frac{\partial_i u\ \partial_j u}{1-|Du|^2}\right)\partial^2_{ij}u~.$$

\subsection{CMC hypersurfaces}

Let us start by recalling some results on entire hypersurfaces of constant mean curvature that will be used in the following. By the Lorentzian Bernstein theorem, the only entire hypersurfaces with vanishing mean curvature are spacelike hyperplanes, so we consider nonzero values $c$ of the (constant) mean curvature. Up to applying a time-reversing isometry, we can assume that the mean curvature is positive. 

First of all, the following result shows, in particular, that the domain of dependence of an entire CMC hypersurface with positive $c$ is a future regular domain.

\begin{thm}[{\cite[Corollary 2.4]{BSS2}}]\label{thm:CMC reg dom}
Let $\Sigma$ be an entire spacelike hypersurface with mean curvature bounded below by a positive constant. Then $\mathcal D(\Sigma)$ is a future regular domain. 
\end{thm}

By Theorem \ref{thm:CMC reg dom}, 
in order to obtain a classification result for entire CMC hypersurfaces, it it sufficient to consider future regular domains. The classification result is the following.

\begin{thm}[{\cite[Theorem A]{BSS2}}]\label{thm:existence CMC}
Given any future regular domain $\mathcal D$ and any $c>0$, there exists a unique entire spacelike hypersurface with constant mean curvature $c$ such that  $\mathcal D(\Sigma)=\mathcal D$. 
\end{thm}

The uniqueness part of Theorem \ref{thm:existence CMC} follows from the following comparison principle for mean curvature, that we will apply in this more general form. 

\begin{thm}[{\cite[Theorem 2.1]{BSS2}}]\label{thm:comparison CMC}
Let $\Sigma_-$ and $\Sigma_+$ be entire spacelike hypersurfaces, which are the graphs of $u_-$ and $u_+$ respectively. Suppose that $\Sigma_+$ has constant mean curvature $\mathcal H_1[u_+]=c>0$, and $\Sigma_-$ has (possibly non-constant) mean curvature $\mathcal H_1[u_-]\geq c$, and that $\Sigma_+\subseteq \mathcal D(\Sigma_-)$. Then $u_+(x)\geq u_-(x)$ for every $x\in\R^n$. 
\end{thm}

Let us now briefly describe the relation between entire CMC hypersurfaces and product regular domains. The key result is the following Splitting Theorem:

\begin{thm}[{\cite[Theorem 3.1]{CT}}]\label{splitting theorem}
Let $\Sigma$ be an entire spacelike hypersurface in $\R^{n,1}$ with constant mean curvature $c>0$. Then exactly one of the following holds:
\begin{enumerate}
\item $\Sigma$ is strictly convex, or
\item there exists a subspace $P$ of signature $(k,1)$, for $1\leq k\leq n-1$, such that $\Sigma=\{x+v\,|\,x\in \Sigma',v\in P^\perp\}$, for $\Sigma'\subset P\cong \R^{k,1}$ an entire  hypersurface of constant mean curvature $nc/k$. \label{item split cmc}
\end{enumerate}
\end{thm}

Clearly, the domain of dependence of a CMC hypersurface as in item \eqref{item split cmc} splits (as per Definition \ref{def split domain}). Conversely, by the uniqueness part of Theorem \ref{thm:existence CMC}, if $\mathcal D=\mathcal D(\Sigma)$ splits, then using the invariance of $\mathcal D$ by translations in $P^\perp$, one sees immediately that $\Sigma$ is in the form of item \eqref{item split cmc}. We summarize this discussion in the following corollary:

\begin{cor}\label{cor:splitting}
Let $\Sigma$ be an entire spacelike hypersurface in $\R^{n,1}$ with constant mean curvature $c>0$. Then $\Sigma$ is strictly convex if and only if $\mathcal D(\Sigma)$ does not split.
\end{cor}

\subsection{Scalar curvature and admissible functions}\label{sec:scalar and admissible}

Let us now focus on the study of the scalar curvature (or equivalent, on the $\mathcal H_2$) of spacelike hypersurfaces. From the elementary inequality
\begin{equation}\label{eq:silly inequality}
\left(\sum_{i=1}^n\lambda_i\right)^2\geq 2\sum_{1\leq i<j\leq n}\lambda_i\lambda_j~,
\end{equation}
 it follows that if $\Sigma$ is a (connected) hypersurface of scalar curvature bounded above by a negative constant, then its mean curvature is either bounded below by a positive constant, or bounded above by a negative constant. Up to applying a time-reversal isometry in the latter case, and using Theorem \ref{thm:CMC reg dom}, we obtain the following corollary.

\begin{cor}\label{cor:scalar reg dom}
Let $\Sigma$ be an entire spacelike hypersurface whose scalar curvature is bounded above by a negative constant. Then $\mathcal D(\Sigma)$ is either a future regular domain or a past regular domain. In the former case, the mean curvature of $\Sigma$ is bounded below by a positive constant; in the latter, the mean curvature of $\Sigma$ is bounded above by a negative constant. 
\end{cor}

In order to obtain existence and uniqueness results for entire spacelike hypersurfaces with a prescribed domain of dependence, we will thus focus (which is again not restrictive, up to a time-reversing isometry) on the first situation in Corollary \ref{cor:scalar reg dom}.
It is thus natural to consider the set of \emph{admissible functions}:
$$K_2:=\{u:\R^n\rightarrow\R\ \in C^2(\R^n)|\ |Du|<1,\ \mathcal{H}_1[u]>0\mbox{ and }\mathcal{H}_2[u]>0\mbox{ on }\R^n\}.$$
{On $K_2$ the operator $\mathcal{H}_2$ is elliptic, the operator $\mathcal{H}_2^{\frac{1}{2}}$ is concave with respect to the second derivatives, and the Maclaurin inequality (which is stronger than \eqref{eq:silly inequality}) holds:
\begin{equation}
0<{\mathcal{H}_2[u]}^{\frac{1}{2}}\leq \mathcal{H}_1[u]\ .
\end{equation}
Details are given in Appendix \ref{appendix Hm}.}

Let us state the standard comparison principles for the curvature operators $\mathcal{H}_1$ and $\mathcal{H}_2$:
\begin{thm}\label{standard comp principle H1}
Let $\Omega$ be a bounded open subset of $\R^n$ and $u,v\in C^2(\Omega)\cap C^0(\overline{\Omega})$ be spacelike functions such that $\mathcal{H}_1[u]\geq\mathcal{H}_1[v]$ in $\Omega$ and $u\leq v$ on $\partial\Omega$. Then $u\leq v$ on $\overline{\Omega}.$ 
\end{thm}
\begin{thm}\label{standard comp principle H2}
Let $\Omega$ be a bounded open subset of $\R^n$ and $u,v\in C^2(\Omega)\cap C^0(\overline{\Omega})$ be spacelike functions such that $\mathcal{H}_2[u]\geq\mathcal{H}_2[v]$ in $\Omega$ and $u\leq v$ on $\partial\Omega$. Assume moreover that $u$ is admissible. Then $u\leq v$ on $\overline{\Omega}.$ 
\end{thm}

In both statements we also have by the strong maximum principle: if $\Omega$ is connected, then either $u<v$ in $\Omega$ or $u$ and $v$ coincide.

\begin{rem}\label{rmk:geometric max principle tangent}
Theorems \ref{standard comp principle H1} and \ref{standard comp principle H2} are equivalent to the following statements, that we record here since they will be useful later. Let $u,v\in C^2(\Omega)$, for $\Omega$ an open subset of $\R^n$, and let $x_0\in\Omega$. Suppose $u(x_0)=v(x_0)$ and $u\leq v$ on $\Omega$. Then $\mathcal{H}_1[u](x_0)\leq\mathcal{H}_1[v](x_0)$. If moreover $u$ is admissible, then $\mathcal{H}_2[u](x_0)\leq\mathcal{H}_2[v](x_0)$.
\end{rem}

% !TEX root = main.tex

\section{Existence of entire solutions}\label{section existence}

\subsection{Outline of the construction}
Let us provide an outline of the strategy to prove the existence parts of Theorem \ref{thm existence} and Theorem \ref{thm existence 2}. Given a future regular domain $\mathcal{D}_{\varphi}$ for which the cardinality of $F_\varphi$ is at least three, 
we will construct two barriers $\underline{u},\overline{u}:\R^n\rightarrow\R$ whose graphs have domain of dependence $\mathcal{D}_{\varphi}.$ We will take for the upper barrier $\overline{u}$ a solution of the prescribed constant mean curvature equation obtained in \cite{BSS2}, and for the lower barrier $\underline{u}$ a supremum of solutions of the prescribed constant scalar curvature equation: the latter solutions will be products with linear spaces of surfaces with constant Gauss curvature in $\R^{2,1}$ and triangular Gauss map image obtained in \cite{BSS}. We then solve a sequence of Dirichlet problems between the barriers, and extract a convergent sub-sequence using a priori estimates essentially obtained in  \cite{Ba1} and \cite{Ur}.
 
A technical, but important, point is the following. We will need to use that the upper barrier $\overline u$ is strictly convex, which by the Splitting Theorem (see Corollary \ref{cor:splitting}) is the case if and only if the set $F_\varphi\subset \mathbb{S}^{n-1}$ does not belong to any affine hyperplane of $\R^n$ (i.e. if and only if $\mathcal{D}_{\varphi}$ does not split, see Proposition \ref{prop:char split domain}). While this is an hypothesis of Theorem \ref{thm existence 2}, for Theorem \ref{thm existence} we will use the following remark: if $\mathcal{D}_{\varphi}$ splits as $\mathcal{D}_{\varphi}'\times\R^{n-k}$ with $\mathcal{D}_{\varphi}'\subset\R^{k,1},$ a solution $\Sigma'$ of constant scalar curvature in $\R^{k,1}$ and domain of dependence $\mathcal{D}_{\varphi}'$ yields a hypersurface $\Sigma:=\Sigma'\times\R^{n-k}$ in $\R^{n,1}$ with constant scalar curvature and domain of dependence $\mathcal{D}_{\varphi};$ so, in that case, the existence of a solution readily follows from the existence of a solution in a space of smaller dimension. The proof of Theorem \ref{thm existence} will be concluded by an inductive argument, see Section \ref{subsec:conclusion}.

In this section, we construct  entire admissible solutions of $\mathcal{H}_2[u]=H(.,u)$ with domain of dependence $\mathcal{D}_{\varphi}$. For this purpose, we first construct the barriers $\underline{u}$ and $\overline{u}$ in Section \ref{section barriers}, then we solve the Dirichlet problem between the barriers (see Appendix \ref{section Dirichlet}) and we reduce the construction of entire solutions to the obtention of $C^1$ and $C^2$ interior estimates in Section \ref{section construction entire}. Then we obtain these estimates in Sections \ref{subsec C1 estimate}  and \ref{subsec C2 estimate} and we conclude the proofs of the existence parts of Theorems \ref{thm existence} and \ref{thm existence 2} in Section \ref{subsec:conclusion}.

\subsection{The construction of the barriers}\label{section barriers}

\begin{defn}\label{defi barriers}
Let $\varphi:\mathbb S^{n-1}\to\R\cup\{+\infty\}$ be a lower semi-continuous function such that $card(F_\varphi)\geq 3$, and let $h_0,h_1$ be constants such that $0<h_0\leq h_1$. We say that $\underline u,\overline u:\R^n\to\R$ is a pair of $(\mathcal D_\varphi,h_0,h_1)$-barriers if:
\begin{enumerate}
\item $\overline u$ is smooth, spacelike and strictly convex; \label{barriers item over}
\item $\underline u$ is 1-Lipschitz; \label{barriers item under}
\item The domain of dependence of the graphs of $\underline u$ and $\overline u$ is $\mathcal D_\varphi$
and the following inequalities hold:
\begin{equation}\label{ineg bars}
V_{\varphi}<\underline{u}<\overline{u}<V_{\varphi}+C_0~,
\end{equation}
for some $C_0>0$;
\label{barriers item inequalities}
\item For every bounded domain $\Omega\subset\R^n$, every admissible function $u:\overline \Omega\to\R$ such that $h_0\leq \mathcal H_2[u]\leq h_1$ and  $\underline u|_{\partial \Omega}\leq u|_{\partial \Omega}\leq \overline u|_{\partial \Omega}$, satisfies $\underline u\leq u\leq \overline u$ on $\Omega$.
\label{barriers item comparison}
\end{enumerate}
\end{defn}

\begin{rem}\label{rmk after definition barriers}
In item \eqref{barriers item inequalities}, the hypothesis that the domain of dependence of the graph of $\overline u$ equals $\mathcal D_\varphi$, together with \eqref{ineg bars}, automatically implies that the domain of dependence of the graph of $\underline u$ also equals $\mathcal D_\varphi$. 
\end{rem}

\begin{rem}
Classically, the upper and lower barriers are taken as two admissible functions satisfying $\mathcal H_2[\overline u]\leq h_0$ and $\mathcal H_2[\underline u]\geq h_1$, so that item \eqref{barriers item comparison} is satisfied as a consequence of Theorem \ref{standard comp principle H2}. In fact, in our Definition \ref{defi barriers} we require that $\overline u$ is smooth, and under this assumption the condition \eqref{barriers item comparison} actually implies that $\mathcal H_2[\overline u]\leq h_0$. However, in this work we will need to construct a function $\underline u$ which is possibly not smooth, obtained as a supremum of admissible functions satisfying a lower bound on $\mathcal H_2$. 
\end{rem}

\begin{prop}\label{prop barriers exist}
Let $\varphi:\mathbb S^{n-1}\to\R\cup\{+\infty\}$ be a lower semi-continuous function such that $card(F_\varphi)\geq 3$, and such that $F_\varphi$ is not contained in any affine hyperplane of $\R^{n}$. Then, for any constants $0<h_0\leq h_1$, a pair of $(\mathcal D_\varphi,h_0,h_1)$-barriers exists.
\end{prop}
\begin{proof}
Let $\alpha\leq \sqrt{h_0}$. (We can take $\alpha= \sqrt{h_0}$ for the moment but, if $h_0=h_1$, then we will have to replace $\alpha$ by a smaller constant later on.)

To construct $\overline{u}$ as in item \eqref{barriers item over}, let us take  the function whose graph has constant mean curvature $\mathcal{H}_1=\alpha$ and whose domain of dependence is $\mathcal{D}_{\varphi},$ given in \cite{BSS2}. This upper barrier $\overline{u}$ is strictly convex by Proposition \ref{prop:char split domain} and Corollary \ref{cor:splitting}. 

Let us construct $\underline u$ as in item \eqref{barriers item under}. Consider a subset $T=\{y_a,y_b,y_c\}$ of $F_\varphi\subset\mathbb{S}^{n-1}\subset\R^n$ formed by three pairwise distinct points. The null vectors
$$\overrightarrow{0y_a}=(y_a,1),\ \overrightarrow{0y_b}=(y_b,1),\ \overrightarrow{0y_c}=(y_c,1)$$
of $\R^{n,1}$ span the 3-dimensional linear space
$$L_T:=\mathrm{Span}(\overrightarrow{0y_a},\overrightarrow{0y_b},\overrightarrow{0y_c})$$
and we have the decomposition $\R^{n,1}=L_T\oplus  L_T^{\perp}.$ The restriction to $L_T$ of the metric in $\R^{n,1}$ is of signature $(2,1)$ and the null planes
$$P_a={\overrightarrow{0y_a}}^\perp-\varphi(a),\ P_b={\overrightarrow{0y_b}}^\perp-\varphi(b),\ P_c={\overrightarrow{0y_c}}^\perp-\varphi(c)$$
define a triangular regular domain
$$\mathcal{D}_T=I^+(P_a)\cap I^+(P_b)\cap I^+(P_c)$$
in $L_T;$ here we use $I^+(P)$ to denote the future half-space whose boundary is the null plane $P$. By \cite[Theorem A]{BSS} there exists a spacelike surface $\Sigma_T$ in $L_T$ with constant Gauss curvature $K=-\frac{n(n-1)}{2}h_1$ and whose domain of dependence is $\mathcal{D}_T.$ Let us note that $\Sigma_T$ is asymptotic at infinity to the boundary $\partial \mathcal D_T$ of its domain of dependence. The product $\Sigma_T\times L_T^{\perp}$ then defines an admissible hypersurface of $\R^{n,1}=L_T\oplus L_T^{\perp}$ with constant scalar curvature $\mathcal{H}_2=h_1.$ It is the graph  of an entire function $z_T:\R^n\rightarrow\R.$ Denoting by $\mathcal{T}$ the set of triples $T=\{y_a,y_b,y_c\}$ of pairwise distinct points in $F_\varphi$ we finally set
\begin{equation}\label{eq:defi underline u}
\underline{u}=\sup_{T\in\mathcal{T}}z_T.
\end{equation}

Let us now show item \eqref{barriers item inequalities}.
The domain of dependence of the graph of $\overline u$ is $\mathcal{D}_{\varphi}$ by construction. By Remark \ref{rmk after definition barriers}, the same holds true for $\underline u$ if \eqref{ineg bars} holds. We thus only have to show \eqref{ineg bars}, namely that $V_{\varphi}<\underline{u}<\overline{u}<V_{\varphi}+C_0.$

We first show that $V_{\varphi}<\underline{u}.$ Let us set, for each $T\in \mathcal{T},$ 
$$V_T(x):=\sup_{y\in T}\left(\langle x,y\rangle-\varphi(y)\right).$$ 
Recalling that
$$V_{\varphi}(x)=\sup_{y\in F_{\varphi}}\left(\langle x,y\rangle-\varphi(y)\right)$$
and $F_\varphi=\cup_{T\in\mathcal{T}}T$, we have
\begin{equation}\label{eq:ineq V_T}
V_{\varphi}=\sup_{T\in\mathcal{T}}V_T~.
\end{equation}

 Since $z_T>V_T$ for all $T\in\mathcal{T},$ from \eqref{eq:defi underline u} and \eqref{eq:ineq V_T} we deduce that $\underline{u}\geq V_{\varphi}.$ The inequality is strict: let us fix $x\in\R^n;$ since $\varphi$ is lower semi-continuous, there exists $y_0\in F_\varphi$ such that
\begin{equation}\label{def y_0}
V_{\varphi}(x)=\sup_{y\in F_\varphi}\left(\langle x,y\rangle-\varphi(y)\right)=\langle x,y_0\rangle-\varphi(y_0).
\end{equation}
If $T=\{y_0,y_1,y_2\}$ is a triple containing $y_0,$ with arbitrary other points $y_1,y_2\in F_\varphi$ (that exist since $card(F_\varphi)\geq 3$), we have, by definition of $y_0,$
\begin{eqnarray*}
V_{T}(x)&=&\max\left(\langle x,y_0\rangle-\varphi(y_0),\langle x,y_1\rangle-\varphi(y_1),\langle x,y_2\rangle-\varphi(y_2)\right)\\
&=&\langle x,y_0\rangle-\varphi(y_0).
\end{eqnarray*}
Hence $V_{\varphi}(x)=V_T(x)$ for all $T\subset F_\varphi$ containing the point $y_0$ defined by (\ref{def y_0}). We thus have
$$V_{\varphi}(x)=V_T(x)<z_T(x)\leq\underline{u}(x),$$
and thus $V_{\varphi}<\underline{u}.$

Second, we show that $\underline{u}<\overline{u}.$ By Theorem \ref{thm:comparison CMC} applied to $u_-=z_T$ and $u^+=\overline u$, since $\mathcal H_1[\overline u]=\alpha$ and $\mathcal H_1[z_T]\geq \mathcal H_2[z_T]^{1/2}=h_1^{1/2}\geq \alpha$, we have that $z_T\leq\overline{u}$ for all $T\in\mathcal{T}.$ Taking the supremum, we deduce that $\underline{u}\leq \overline{u}$ on $\R^n.$ We may moreover suppose that the strict inequality holds, up to replacing $\alpha$ by a slightly smaller constant. Indeed, if $\overline{v},$ constructed by Theorem \ref{thm:existence CMC}, has constant mean curvature $\alpha'=\mathcal{H}_1[\overline{v}]<\mathcal{H}_1[\overline{u}]=\alpha$, then $\overline{u}\leq \overline{v}$ by Theorem \ref{thm:comparison CMC} and moreover $\overline{u}<\overline{v}$ by the strong maximum principle; replacing $\overline{u}$ by $\overline{v}$ if necessary, we may therefore suppose that $\underline{u}< \overline{u}.$

Third, we show that $\overline{u}<V_{\varphi}+C_0$.
For this, we use the notion of cosmological time (see \cite[Section 4]{Bo} for some foundational properties) in the regular domain $\mathcal D_\varphi$, which is the function 
$T:\mathcal D_\varphi\to(0,+\infty)$
such that $T(p)$ is the supremum of the length of every past-directed causal curve $\gamma:[0,a]\to \mathcal D_\varphi$ with $\gamma(0)=p$, where past-directed means that $\langle \gamma'(t),\partial_{x_{n+1}}\rangle>0$ and the length of $\gamma$ is:
$$\ell(\gamma)=\int_0^a \sqrt{|\langle \gamma'(t),\gamma'(t)|}dt~.$$ 
By \cite[Lemma 2.3]{BSS2}, the cosmological time $T$ restricted to the graph of $\overline u$ is bounded from above by $C_0=1/\alpha$. 
This implies the desired inequality. Indeed, the cosmological time of the points on the graph of the function $V_{\varphi}+C_0$ is at least $C_0$, since the vertical segment connecting the points $(x,V_\varphi(x))$ and $(x,V_\varphi(x)+C_0)$ is a timelike curve of length $C_0$.
This finishes the proof of item \eqref{barriers item inequalities}.

We finally prove item \eqref{barriers item comparison}. First,  since $\mathcal{H}_1[u]\geq \mathcal{H}_2^{\frac{1}{2}}[u]\geq h_0^{1/2}\geq \alpha=\mathcal{H}_1[\overline u]$, we have $u\leq\overline{u}$ by the comparison principle (Theorem \ref{standard comp principle H1}) for the mean curvature on the bounded domain $\Omega$. Second, for every function $z_T$ in the construction of $\underline u$, we have $\mathcal{H}_2[u]\leq h_1=\mathcal{H}_2[z_T]$, hence by the comparison principle for $\mathcal H_2$ (Theorem \ref{standard comp principle H2}) we obtain $u\geq z_T$. Taking the supremum over all $T$ we conclude that  $u\geq\underline{u}$.
\end{proof}

\subsection{Construction of a solution}\label{section construction entire}
In the following, we suppose that $F_{\varphi}$ is not included in any affine hyperplane of $\R^n$ and fix $h_0,h_1$ positive constants such that $h_0\leq H\leq h_1.$
Moreover, $\underline u$ and $\overline u$ will denote a pair of $(\mathcal D_\varphi,h_0,h_1)$-barriers as in Definition \ref{defi barriers}, constructed in Proposition \ref{prop barriers exist}. 

Consider the balls $B_i:=\{x\in\R^n|\ |x|<i\}$ for $i\in\N^*$.
We have $\R^n=\cup_{i\in\N^*} B_i$ with $B_i\subset B_{i+1}.$ Let $u_i\in C^{4,\alpha}(\overline{B_i})$ be a solution of the Dirichlet problem 
$$\mathcal{H}_2[u_i]=H(.,u_i)\mbox{ in } B_i\ \mbox{ and }\ u_i=\overline{u}\mbox{ on }\partial B_i$$
such that $\underline{u}\leq u_i\leq \overline{u}$. We have included an outline of the proof of the solution of the Dirichlet problem in Appendix \ref{section Dirichlet}. 
The solution $u_i$ is obtained by applying Theorem \ref{thm DP barriers} to $\varphi_2=\overline u$ and $\varphi_1$  some subsolution (see Remark \ref{rem appendix varphi1}). Theorem \ref{thm DP barriers} provides a solution such that $\varphi_1\leq u_i\leq\varphi_2=\overline u$, and we actually have $\underline u\leq u_i$ by item \eqref{barriers item comparison} of Definition \ref{defi barriers}.

To construct an entire solution we need the following interior estimates: if $K\subset \R^n$ is a compact subset, there exist $I_K\in\N,$ $\theta_K\in (0,1]$ and $C_K\geq 0$ such that, for all $i\geq I_K,$
\begin{equation}\label{C1 K estimate i}
\sup_{K}|Du_i|\leq 1-\theta_K
\end{equation}
and
\begin{equation}\label{C2 K estimate i}
\sup_{K}|D^2u_i|\leq C_K.
\end{equation}
With these estimates at hand, a $C^{2,\alpha}$ estimate is obtained using the Evans-Krylov theory {(note that the $C^0$ estimate is trivial since $\underline{u}\leq u_i\leq \overline{u}$)}, and an entire solution is then obtained using the Arzel\`a-Ascoli theorem together with a diagonal process.

The interior $C^1$ and $C^2$ estimates (\ref{C1 K estimate i}) and (\ref{C2 K estimate i}) were obtained in \cite{Ba1} and \cite{Ur} once two auxiliary functions are constructed, as explained in the following subsections.
\subsection{The interior $C^1$ estimate}\label{subsec C1 estimate}
The estimate relies on the following result: 
\begin{thm}[{\cite[Section 4]{Ba1}}]\label{thm C1 K estimate}
Let $K$ be a compact subset of $\R^n$ and $R>0$ such that $K\subset B_R.$ If there exists a smooth spacelike function $\psi:\overline{B}_R\rightarrow\R$ with the following property:
\begin{equation}\label{eq:aux1}
\psi<\underline{u}\text{ on }K\qquad\text{and}\qquad\psi>\overline{u}\text{ on }\partial B_R \tag{$\star$}
\end{equation}
then there exists $\theta\in (0,1]$ such that every solution $u:B_R\to\R$ of $\mathcal{H}_2[u]=H(.,u)$ with $\underline{u}\leq u\leq\overline{u}$ satisfies
\begin{equation}\label{C1 K estimate}
\sup_{K}|Du|\leq 1-\theta.
\end{equation}
The number $\theta$ depends on $K,R,\underline{u},\overline{u},\psi$ and $H$ on $\{(x,t)|\ x\in B_R,\ \psi(x)\leq t\leq \overline{u}(x)\}.$
\end{thm}

We now construct the required auxiliary function $\psi$, given a pair of $(\mathcal D_\varphi,h_0,h_1)$-barriers. Observe that the hypothesis on the pair of barriers used in the proof are the fact that $\overline u$ is spacelike and the fact that $\underline u>V_\varphi$.

\begin{lem}\label{lem psi exist}
If $K$ is a compact subset of $\R^n,$ there exist $R>0$, with $K\subset B_R$, and a smooth spacelike function $\psi:\overline{B}_R\rightarrow\R$, such that condition \eqref{eq:aux1} holds.
\end{lem}
\begin{proof}
Let us fix $R_0$ such that $K\subset \overline{B}_{R_0}$ and $\psi_0:\overline{B}_{R_0}\rightarrow\R$ be a 1-Lipschitz function such that $\psi_0<\underline{u}$ and whose graph belongs to $\mathcal{D}_\varphi$ (me may take for instance $\psi_0=V_{\varphi}+\alpha_0$ with $\alpha_0>0$ such that $\alpha_0<\inf_{\overline{B}_{R_0}}(\underline{u}-V_\varphi)$). We then consider $\psi:\R^n\rightarrow\R$ such that $\psi=\psi_0$ on $\overline{B}_{R_0}$ and
$$\psi(\xi+t\xi/R_0)=\psi_0(\xi)+t$$
for all $\xi\in\partial\overline{B}_{R_0}$ and $t\geq0$. By construction, for $p=(\xi,\psi_0(\xi))$ and $v=(\xi/R_0,1)$, the null line $p+tv,$ $t\geq 0$ belongs to the graph of $\psi.$ We claim that there exists a (maybe large) $R$ such that $\psi>\overline{u}$ on $\R^n\backslash B_{R}.$ To see this, we apply Lemma \ref{lemme:continuité} to the subset $\{(p=(\xi,\psi_0(\xi)),v=(\xi/R_0,1))\,|\,\xi\in\partial B_{R_0}\}$, which is in $\Xi$ since $(\xi,\psi_0(\xi))\in\mathcal D_\varphi$ (Remark \ref{rem domain omega appendix}). Lemma \ref{lemme:continuité} shows that the set $\{\psi=\overline{u}\}$ is the image of a continuous function from $\partial B_{R_0}$ to the graph of $\overline{u}$, whose image is compact since $\partial B_{R_0}$ is compact. Hence $\{\psi<\overline{u}\}$ is bounded. Finally, we fix $\delta>0$ such that
$$\delta<\min(\inf_{\partial B_{R}}(\psi-\overline{u}),\inf_K(\underline{u}-\psi)),$$
we consider $\rho_{\varepsilon}\in C^{\infty}_c(B_{\varepsilon})$  such that $\rho_{\varepsilon}\geq 0$ and $\int_{\R^n}\rho_{\varepsilon}(z)dz=1$ and choose $\varepsilon$ small so that the function $\psi_\varepsilon=(1-\varepsilon)\psi*\rho_{\varepsilon}$ satisfies $\sup_{\overline{B}_{R}}|\psi_{\varepsilon}-\psi|\leq\delta.$ The function $\psi_{\varepsilon}$ is smooth, and it is spacelike since
\begin{eqnarray*}
|\psi_\varepsilon(x)-\psi_\varepsilon(y)|&=&\left|(1-\varepsilon)\int_{\R^n}(\psi(x-z)-\psi(y-z))\rho_\varepsilon(z)dz\right|\\
&\leq& (1-\varepsilon)|x-y|
\end{eqnarray*}
for all $x,y\in\R^n$ (since $\psi$ is $1-$Lipschitz). It is moreover such that $\psi_\varepsilon<\underline{u}$ on $K$ and $\psi_{\varepsilon}>\overline{u}$ on $\partial B_{R}.$
\end{proof}

\subsection{The interior $C^2$ estimate}\label{subsec C2 estimate}
The estimate is a consequence of the following $C^2$ estimate of Urbas:
\begin{thm}[{\cite{Ur}}] \label{thm:C^2 estimates urbas}
Let $K$ be a compact subset of $\R^n$ and $R>0$ such that $K\subset B_R.$ If there exists a smooth and strictly convex function $\phi:\overline{B}_R\rightarrow\R$ 
with the following property:
\begin{equation}\label{eq:aux2}
\phi>\overline{u}\text{ on }K\qquad\text{and}\qquad\phi<\underline{u}\text{ on }\partial B_R~, \tag{$\star\star$}
\end{equation}
then there exists $C>0$ such that every solution $u:B_R\to\R$ of $\mathcal{H}_2[u]=H(.,u)$ with $\underline{u}\leq u\leq\overline{u}$ satisfies
\begin{equation}\label{C2 K estimate}
\sup_{K}|D^2u|\leq C.
\end{equation}
The constant $C$ depends on $K,R,\underline{u},\overline{u},\phi,H$ on $\{(x,t)|\ x\in B_R,\ \underline{u}(x)\leq t\leq \phi(x)\}$ and on $\theta\in (0,1]$ such that $\sup_{B_R}|Du|\leq 1-\theta.$
\end{thm}
Note that the existence of a controlled constant $\theta\in (0,1]$ such that $\sup_{B_R}|Du|\leq 1-\theta$ is granted for every solution $u:B_{R'}\rightarrow\R$ between the barriers by the interior $C^1$ estimate obtained in the previous section, if $R'>R$ is sufficiently large (Theorem \ref{thm C1 K estimate} and Lemma \ref{lem psi exist}). 

We now construct the auxiliary function $\phi$. Here the key property of the barriers is the strict convexity of $\overline u$ (for which the hypothesis that $F_\varphi$ is not contained in any affine hyperplane of $\R^{n}$ is essential) and the inequalities in \eqref{ineg bars}.

\begin{lem}\label{lem phi exist}
If $K$ is a compact subset of $\R^n,$ there exist $R>0$, with $K\subset B_R$, and a smooth and strictly convex function $\phi:\overline{B}_R\rightarrow\R$, such that condition \eqref{eq:aux2} holds.
\end{lem}
\begin{proof}
Applying an isometry of $\R^{n,1}$ we can assume that $\overline u(0)=0$ and $D\overline u_0=0$, which, by strict convexity of $\overline u$, implies 
\begin{equation}\label{est C2 ub tends infty}
\lim_{|x|\rightarrow+\infty}\overline{u}(x)=+\infty~.
\end{equation}
The proof is then analogous to the proof of \cite[Lemma 3.7]{Ba2}. We fix $R'$ sufficiently large such that $K\subset \overline{B}_{R'}.$ We set $\phi_0:=\sup_{B_{R'}}\overline{u}+1.$ Recalling (\ref{ineg bars}) we have $\underline{u}\geq\overline{u}-c$ on $\R^n.$ We thus get from (\ref{est C2 ub tends infty}) the existence of $R>R'$ such that
 $$\inf_{\{x:\ |x|\geq R\}}\underline{u}(x)> \phi_0+1.$$
 We set, for all $x\in\R^n$, 
$$\phi(x):=\phi_0+\frac{1}{R^2}|x|^2.$$
 The function $\phi$ is strictly convex, $\phi\geq \overline{u}+1$ on $B_{R'}$ and $\phi<\underline{u}$ on $\partial B_R.$
\end{proof}

\begin{rem}
Up to taking a larger value of $R$, we could assume that the function $\phi$ constructed in the proof is of spacelike type on $B_R$. However, this is not necessary for the $C^2$ estimates (see Theorem \ref{thm:C^2 estimates urbas}).
\end{rem}

\subsection{Conclusion of the proofs}\label{subsec:conclusion}

\begin{proof}[Proof of existence part of Theorem \ref{thm existence 2}]
By Proposition \ref{prop barriers exist}, let $\underline u$ and $\overline u$ be a pair of $(\mathcal D_\varphi,h_0,h_1)$-barriers, according to Definition \ref{defi barriers}, which exist because we are assuming that $F_\varphi$ is not contained in any affine hyperplane of $\R^n$.

Let $B_i,$ $i\in\N^*,$  be an increasing sequence of balls (of radius $i$, say). From Theorem \ref{thm DP barriers} in Appendix \ref{section Dirichlet}, let $u_i\in C^{4,\alpha}(\overline{B_i})$ be an admissible solution of the Dirichlet problem 
$$\begin{cases} \mathcal{H}_2[u_i]=H(\cdot,u_i)&\mbox{ in } B_i\\ u_i=\overline{u}&\mbox{ on }\partial B_i\end{cases}$$
such that $\varphi_1\leq u_i\leq\varphi_2:=\overline{u}$, where $\varphi_1$ is any subsolution (see Remark \ref{rem appendix varphi1}).
We then have $\underline u\leq u_i\leq \overline u$ by item \eqref{barriers item comparison} of Definition \ref{defi barriers}. 

We have obtained in Sections \ref{subsec C1 estimate} and \ref{subsec C2 estimate} the following local $C^2$ estimates: for every compact subset $K\subset \R^n$ there exist $I_K\in\N,$ $\theta_K\in (0,1]$ and $C_{2,K}\geq 0$ such that
$$\sup_K|Du_i|\leq 1-\theta_K\hspace{.5cm} \mbox{and}\hspace{.5cm}\|u_i\|_{C^2(K)}\leq C_{2,K}$$
for all $i\geq I_K,$ where $\|.\|_{C^2(K)}$ stands for the usual $C^2$ norm on the compact $K.$ These estimates control the ellipticity of the equation, see Proposition \ref{prop ell unif appendix Hm} in Appendix \ref{appendix Hm}. If $K$ is an arbitrary compact subset of $\R^n,$ we may thus assume that the $u_i$ are solutions of an uniformly elliptic equation on $K$, and we deduce from the Evans-Krylov $C^{2,\alpha}$ estimate the following local $C^{2,\alpha}$ estimate (see e.g. \cite[Theorem 17.14]{GT}): for every compact subset $K\subset \R^n$ there exist $\alpha_K\in (0,1),$ $I_K\in\N$ and $C'_{2,K}\geq 0$ such that
$$\|u_i\|_{C^{2,\alpha_K}(K)}\leq C'_{2,K}$$
for all $i\geq I_K,$ where $\|.\|_{C^{2,\alpha_K}(K)}$ denotes the usual $C^{2,\alpha_K}$ norm on $K.$ Let us note that, in order to apply the Evans-Krylov $C^{2,\alpha}$ estimate, we also use here that the operator $\mathcal{H}_2^{1/2}$ is a concave function of the second derivatives on the range of an admissible function, which is property (\ref{Hm concave}) in Appendix \ref{appendix Hm}. The Ascoli-Arzel\`a theorem and a standard diagonal process then yield a subsequence, still denoted $u_i$, which converges in the $C^2$ norm on compact subsets of $\R^n.$ The limit $u:\R^n\rightarrow\R$ is a solution of the equation $\mathcal{H}_2[u]=H(.,u)$, it is admissible since so are the functions $u_i$ and by \eqref{eq:silly inequality} (if $\mathcal{H}_1[u]\geq 0$ and $\mathcal{H}_2[u]>0$ then $\mathcal{H}_1[u]>0$), and finally $u$ is $C^{4,\alpha}$ by the elliptic regularity theory (since $H$ is assumed to be $C^{2,\alpha}$). If moreover $H$ is $C^{k,\alpha}$ for $k>2$, then $u$ is $C^{k+2,\alpha}$, and if $H$ is $C^{\infty}$ then so too is $u$.\\

Let $\Sigma$ be the graph of $u$. Since, for every $i$, $\underline u\leq u_i\leq \overline u$, then for the limit $u$ we still have $\underline u\leq u\leq \overline u$. By the definition of barriers, the graphs $\underline \Sigma$ of $\underline u$ and $\overline \Sigma$ of $\overline u$ are entire hypersurfaces whose domain of dependence is $\mathcal D_\varphi$. Hence any inextensible causal curve in $\mathcal D_\varphi$ intersects both $\underline \Sigma$ and $\overline \Sigma$, and therefore also intersects $\Sigma$. This shows that $\mathcal D_\varphi \subseteq \mathcal D(\Sigma)$. Conversely, if $p\in\R^{n,1}$ does not belong to $\mathcal D_\varphi $, then an inextensible causal curve through $p$ does not intersect $\underline \Sigma$, and therefore it does not intersect $\Sigma$. This shows that $\mathcal D(\Sigma)=\mathcal D_\varphi$ and concludes the proof.
\end{proof}

\begin{proof}[Proof of existence part of Theorem \ref{thm existence}]
We provide the proof by induction on the dimension. First, if $n=2$, the result has been proved in \cite[Theorem A]{BSS}. Suppose now that the result is true for $n\leq n_0$, and we shall prove it for $n=n_0+1$. Let $\varphi:\mathbb{S}^{n_0}\rightarrow\R\cup\{+\infty\}$ be a lower semi-continuous function with $card(F_\varphi)\geq 3$, and let $\mathcal D_\varphi$ be the associated regular domain in $\R^{n_0+1,1}$. If $F_\varphi$ is contained in an affine hyperplane, i.e. if $\mathcal D_\varphi$ splits (Definition \ref{def split domain} and Proposition \ref{prop:char split domain}), let $A$ be an affine subspace of minimal dimension $k\leq n_0$ containing $F_\varphi$, and let $P\cong\R^{k,1}$ be the corresponding linear subspace constructed as in the proof of Proposition \ref{prop:char split domain}. Then $\Sigma=\{x+v\,|\,x\in\Sigma',v\in P^\perp\}$, where $\Sigma'$ is a hypersurface with $\mathcal H_2\equiv c$ in the regular domain $\mathcal D_\varphi\cap P$, {satisfies $\mathcal H_2\equiv ck(k-1)/(n_0+1)n_0$} and has domain of dependence $\mathcal D_\varphi$. If instead $F_\varphi$ is not contained in any affine subspace of dimension $k\leq n_0$, then the conclusion follows as a special case of Theorem \ref{thm existence 2}, for $H$ the constant function.
\end{proof}

% !TEX root = main.tex

\section{Uniqueness and foliation}

\subsection{Uniqueness}

In this section we will prove the following theorem, and we deduce the uniqueness parts of Theorems \ref{thm existence} and \ref{thm existence 2}. 

\begin{thm}\label{thm comparison for scalar curvature at infinity}
Let $\Sigma_-$ and $\Sigma_+$ be entire spacelike hypersurfaces, which are the graphs of two functions $u_-$ and $u_+$ respectively, with $u_-$ admissible. Suppose that $\Sigma_+\subseteq \mathcal D(\Sigma_-)$ and $\mathcal H_2[u_+]\leq \mathcal H_2[u_-]$. Assume moreover that $\mathcal D(\Sigma_+)=\mathcal D_{\varphi_+}$ for a continuous function $\varphi_+$. Then $u_+(x)\geq u_-(x)$ for every $x\in\R^n$. 
\end{thm}

\begin{proof}
Let $\epsilon>0$. We claim that $u_+(x)+\epsilon\geq u_-(x)$ for all $x\in\R^n$. For this purpose, we only need to claim that the subset $\{x\in\R^n\,|\,u_+(x)+\epsilon\leq u_-(x)\}$ is compact, so that we can apply Theorem \ref{standard comp principle H2} to $\Omega:=\{x\in\R^n\,|\,u_+(x)+\epsilon< u_-(x)\}$ and to the functions $u=u_-$ and $v=u_++\epsilon$, which agree on $\partial\Omega$, to deduce that $\Omega$ is actually empty.

To prove the claim, note first that, since $u_+> V_{\varphi_+}$, the set $\{u_++\epsilon\leq u_-\}$ is contained in the set $V:=\{V_{\varphi_+}+\epsilon\leq u_-\}$. Hence we will show that the latter is compact. For this, observe that every ray (say, $R$) of the form $t\mapsto (ty,t-\varphi_+(y)+\epsilon)$, for $y\in \mathbb S^{n-1}$, intersects $\Sigma_-$. Indeed, by Proposition \ref{prop recover varphi} there exists $x\in\R^n$ such that $\langle x,y\rangle-u_+(x)\geq \varphi_+(y)-\epsilon$, hence $u_+(x)\leq\langle x,y\rangle-\varphi_+(y)+\epsilon$. This shows that the null hyperplane $P$ containing $R$, namely $P=\{x_{n+1}=\langle x,y\rangle-\varphi_+(y)+\epsilon\}$, intersects $\Sigma_+$. This implies that the null ray  $R$ intersects $\Sigma_+$, see \cite[Proposition 1.18]{BSS2} or \cite[Section 2.3]{BS}. Since $\Sigma_+\subseteq \mathcal D(\Sigma_-)$, this implies that the ray $R$ intersects $\Sigma_-$ as well.

We now consider the function from $\mathbb S^{n-1}$ to $\Sigma_-$ that associates to every $y\in\mathbb S^{n-1}$ the unique intersection point of the ray $t\mapsto (0,-\varphi_+(y)+\epsilon)+t(y,1)$ with $\Sigma_-$. Applying Lemma \ref{lemme:continuité}, this function is continuous by continuity of $\varphi_+$. Hence its image is compact in $\Sigma_-$. By the definition \eqref{defVphi} of $V_{\varphi_+}$, when $x=ty$, $V_{\varphi_+}(x)+\epsilon\geq \langle x,y\rangle-\varphi_+(y)+\epsilon=t-\varphi_+(y)+\epsilon$, and the latter is the height function along the ray $R$. Hence $V=\{V_{\varphi_+}+\epsilon<u_-\}$ is bounded and this concludes the claim.

Since $\epsilon$ was arbitrary, the conclusion follows by taking the limit as $\epsilon\to 0$.
\end{proof}

We can now conclude the proofs of the uniqueness parts of Theorems \ref{thm existence} and \ref{thm existence 2}.

\begin{proof}[Proof of the uniqueness part of Theorem \ref{thm existence}]
Let $\varphi:\mathbb S^{n-1}\to\R$ be a continuous function, and let $\Sigma_1=graph(u_1)$ and $\Sigma_2=graph(u_2)$ be two entire spacelike hypersurfaces such that $\mathcal D(\Sigma_1)=\mathcal D(\Sigma_2)=\mathcal D_{\varphi}$ and $\mathcal H_2[u_1]=\mathcal H_2[u_2]=c$. Then, applying  Theorem \ref{thm comparison for scalar curvature at infinity} twice, $u_1=u_2$ and thus $\Sigma_1=\Sigma_2$. 
\end{proof}

\begin{proof}[Proof of the uniqueness part of Theorem \ref{thm existence 2}]
We reason as in the proof of Theorem \ref{thm comparison for scalar curvature at infinity}. Let $u_1$ and $u_2$ be admissible functions such that $\mathcal H_2[u_i]=H(\cdot,u_i)$ on $\R^n$ and having the same domain of dependence $\mathcal D_\varphi$, for $\varphi$ continuous.  Fixing $\epsilon>0$, the proof of Theorem \ref{thm comparison for scalar curvature at infinity} shows that $\Omega:=\{x\in\R^n\,|\,u_2(x)+\epsilon< u_1(x)\}$ is bounded. Hence we can apply the comparison principle for scalar curvature (Theorem \ref{standard comp principle H2}) to $\Omega$ and to the functions $u=u_1$ and $v=u_2+\epsilon$. Using that $\partial_{x_{n+1}}H\geq 0$, we see that
$\mathcal H_2[v]=\mathcal H_2[u_2+\epsilon]\leq \mathcal H_2[u_1]=\mathcal H_2[u]$ on $\Omega$. This shows that $\Omega$ is in fact empty, and therefore $u_2+\epsilon\geq u_1$. Letting $\epsilon\to 0$ we obtain that $u_2\geq u_1$, and reversing the roles of $u_1$ and $u_2$ we conclude that $u_1=u_2$.
\end{proof}

\subsection{Foliation}

The existence of a  foliation in Theorem \ref{thm foliation} will be a consequence of the following more general result:
\begin{thm}\label{thm foliation with uniqueness hypothesis}
Let $\varphi:\mathbb{S}^{n-1}\rightarrow\R\cup \{+\infty\}$ be a lower semi-continuous function such that  for all $c>0$ there exists a unique admissible function $u_c:\R^n\rightarrow\R$ such that $\mathcal{H}_2[u_c]=c$ and with domain of dependence $\mathcal{D}_{\varphi}.$ Then there exists a unique foliation of $\mathcal{D}_\varphi$ by hypersurfaces of constant scalar curvature $S\in (-\infty,0).$
\end{thm}

\begin{proof}
Uniqueness is obvious and we focus on the existence part of the statement. Suppose first that the domain $\mathcal D_\varphi$ does not split. Under this assumption, let us note first that if $c_1<c_2$ then $u_{c_1}>u_{c_2}$. Indeed, recall from Subsection \ref{section construction entire} that $u_{c}$ can be constructed as limits of solutions of Dirichlet problems on a ball $B_i$ with boundary values $\overline{u}_c$ on $\partial B_i$. Moreover, from the proof of Proposition \ref{prop barriers exist}, we see that we can take $\overline{u}_c$ to be the hypersurface of constant mean curvature $\sqrt{c}/2$ whose domain of dependence is $\mathcal D_\varphi$. With this choice, we have $\overline{u}_{c_1}>\overline{u}_{c_2}$ by the comparison principle for the mean curvature (Theorem \ref{thm:comparison CMC}). Hence, by the comparison principle for $\mathcal H_2$ (Theorem \ref{standard comp principle H2}) and taking limits, $u_{c_1}\geq u_{c_2}$. Together with the strong maximum principle, we conclude the strict inequality. 

We also have the following properties. First,
$$\lim_{c\rightarrow+\infty}\overline{u}_c =V_\varphi~.$$
Indeed, it was proved in \cite{BSS2} that, for a given regular domain $\mathcal D_\varphi$, the graphs of $\overline u_c$ are entire CMC hypersurfaces that foliate $\mathcal D_\varphi$, and $\overline u_{c_1}>\overline u_{c_2}$ if  $c_1<c_2$. 
This immediately implies:
$$\lim_{c\rightarrow+\infty}u_c=V_\varphi~.$$
Second, 
$$\lim_{c\rightarrow 0}\underline{u}_c=+\infty~.$$
To show this, fix a triplet $T\subset F$. Now, recall that, by construction of $\underline u_c$, we have $\underline u_c\geq z_{T,c}$ where $z_{T,c}$ is an admissible function whose graph is $\Sigma_{T,c}\times L_T^\perp$, and $\Sigma_{T,c}$ is a surface of constant Gaussian curvature $K=-\frac{n(n-1)}{2}c$ in the vector subspace $L_T\cong \R^{2,1}$. By the results of \cite{BSS}, the surfaces $\Sigma_{T,c}$ provide a foliation of the regular domain $D_T$, and therefore $\lim_{c\rightarrow 0}z_{T,c}=+\infty$.  This implies $\lim_{c\rightarrow 0}\underline{u}_c=+\infty$ as claimed. It follows that
$$\lim_{c\rightarrow 0} u_c=+\infty~.$$
We thus only have to prove that any point $X=(x,t)\in\mathcal{D}_{\varphi}$ belongs to the graph of a solution $u_c.$ Setting
$$c^+:=\sup\{c\,|\, u_c(x)>t\}\qquad c^-:=\inf\{c\,|\, u_c(x)<t\}~,$$
we have $c^+\leq c^-$ and we may consider a non-decreasing sequence $c^+_n\rightarrow_{n\rightarrow+\infty} c^+$ and a non-increasing sequence $c^-_n\rightarrow_{n\rightarrow+\infty} c^-$ such that
$$u_{c^+_n}(x)>t>u_{c^-_n}(x)~.$$
Using the estimates obtained in Section \ref{section existence} we may suppose that $u_{c^+_n}$ and $u_{c^-_n}$ converge to functions $u^+$ and $u^-.$ These functions satisfy $u^+(x)\geq t\geq u^-(x)$ and they are solutions of $\mathcal{H}_2[u^+]=c^+$ and $\mathcal{H}_2[u^-]=c^-.$ If $c^+=c^-,$ the uniqueness of a solution implies that $u_+=u_-$ and the result follows. Assuming that $c^+<c^-$ and fixing $c_0$ such that $c^+<c_0<c^-,$ we have $u^-<u_{c_0}<u^+$ and therefore
$$u^-(x)<u_{c_0}(x)<u^+(x).$$ 
The inequality $u_{c_0}(x)>t$ yields a contradiction with the definition of $c^+$ and the inequality $u_{c_0}(x)<t$ a contradiction with the definition of $c^-.$ We deduce that $u_{c_0}(x)=t,$ which finishes the proof in the case where $\mathcal D_\varphi$ does not split. 

If $\mathcal D_\varphi\cong \mathcal D'\times\R^{n-k}$ (see Subsection \ref{subsec:products}) then the uniqueness of the solutions of $\mathcal H_2[u_c]=c$ in $\mathcal D_\varphi$ for every $c>0$ implies, taking products, the uniqueness of the solutions of the same problem in $\mathcal D'$. The existence of the foliation in $\mathcal D_\varphi$ then trivially follows from the first part of the proof applied to $\mathcal D'$, by taking products again.
\end{proof}

\begin{proof}[Proof of Theorem \ref{thm foliation}]
The uniqueness part of Theorem \ref{thm existence geometric}, together with Theorem \ref{thm foliation with uniqueness hypothesis}, immediately implies the existence of a foliation by hypersurfaces of constant scalar curvature as in Theorem \ref{thm foliation}. Moreover, the construction shows that the function that sends $p$ to the value of $S$ such that $p$ is contained in the leaf of the foliation with scalar curvature $S$ is a time function. Indeed, the proof shows that if $c_1<c_2$, then  $u_{c_1}>u_{c_2}$ and the scalar curvature of $u_c$ equals $-n(n-1)c$. This concludes Theorem \ref{thm foliation}.
\end{proof}

Taking products, we obtain immediately that the conclusion of Theorem \ref{thm foliation} also holds for regular domains that split as $\mathcal D_\varphi\times \R^{n-k}$, for $\varphi:\mathbb S^{k-1}\to\R$ continuous. More precisely, using Proposition \ref{prop:char split domain}, the most general statement is the following:

\begin{cor}\label{cor foliation general}
Let $\varphi:\mathbb S^{n-1}\to\R\cup\{+\infty\}$ be a function which is continuous and real-valued on $\mathbb S^{n-1}\cap A$, where $A$ is an affine subspace of $\R^n$ of dimension $3\leq d\leq n$ intersecting $\mathbb S^{n-1}$, and identically equal to $+\infty$ on the complement of $A$. Then $\mathcal D_\varphi$ is foliated by hypersurfaces of constant scalar curvature in $(-\infty,0)$ and the  scalar curvature function associated to this foliation defines a time function.
\end{cor}

% !TEX root = main.tex

\section{Finiteness conditions}

\subsection{Dimension $3+1$}
We prove here that in $\R^{3,1}$ the domain of dependence $\mathcal{D}_{\varphi}$ of an entire admissible solution of $\mathcal{H}_2[u]=1$ is necessarily such that $card(F_\varphi)\geq 3.$ This relies on the following result:
\begin{prop}\label{prop finitness}
There is no entire spacelike function $u:\R^2\rightarrow\R$ such that $\mathcal{H}_1[u]>0$ on $\R^2$ and $\sup_{\R^2}u<+\infty.$
\end{prop}
\begin{proof}
By contradiction, assuming that $u:\R^2\rightarrow\R$ is such that $\mathcal{H}_1[u]>0$ and $\sup_{\R^2}u<+\infty,$ we consider $v:\R^2\setminus\{0\}\rightarrow\R,$ $x\mapsto\log|x|$ and for $\varepsilon>0$ the function $u_{\varepsilon}:=u-\varepsilon v$, which is such that
$$\lim_{|x|\rightarrow+\infty}u_{\varepsilon}(x)=-\infty.$$
Let us first show that 
\begin{equation}\label{ueps x geq 1}
u_{\varepsilon}(x)\leq\max_{|x|=1}u_{\varepsilon}
\end{equation}
for all $x$ such that $|x|\geq 1.$ If not, $u_{\varepsilon}$ would reach its maximum on the set $\{|x|\geq 1\}$ at a point $x_0$ such that $|x_0|>1$, and we would have $Du_{\varepsilon}=0$ and $D^2u_{\varepsilon}\leq 0$ at that point i.e.
$$Du=\frac{\varepsilon}{r}\partial_r\hspace{.5cm}\mbox{and}\hspace{.5cm} D^2u\leq \varepsilon D^2v.$$
Up to applying a rotation, we can assume $x_0=(r,0)$. 
In the following, $u_i=\langle Du,e_i\rangle,$ $u_{ij}=D^2u(e_i,e_j)$ and $v_{ij}=D^2v(e_i,e_j)$ stand for the usual first and second partial derivatives of $u$ and $v$. We thus have $u_1=|Du|$ and $u_2=0$ at $x_0,$ which implies that
\begin{eqnarray*}
\mathcal{H}_1[u]&=&\frac{1}{2\sqrt{1-|Du|^2}}\sum_{1\leq i,j\leq 2}\left(\delta_{ij}+\frac{u_iu_j}{1-|Du|^2}\right)u_{ij}\\
&=&\frac{1}{2\sqrt{1-|Du|^2}}\left(\frac{1}{1-|Du|^2}u_{11}+u_{22}\right)\\
&\leq&\frac{\varepsilon}{2\sqrt{1-|Du|^2}}\left(\frac{1}{1-|Du|^2}v_{11}+v_{22}\right)
\end{eqnarray*}
at $x_0.$ Since $\Delta v=v_{11}+v_{22}=0$ we have $v_{22}=-v_{11}$ and since $v_{11}=-1/r^2<0$ the RHS term is
$$\frac{\varepsilon}{2\sqrt{1-|Du|^2}}\left(\frac{1}{1-|Du|^2}-1\right)v_{11}<0$$
which is impossible since $\mathcal{H}_1[u]>0.$ So (\ref{ueps x geq 1}) holds, which yields, for all $x$ such that $|x|\geq 1,$
$$u(x)\leq \max_{|x|=1} u+\varepsilon\log|x|.$$
Taking the limit as $\varepsilon$ tends to 0, we deduce that for all $x$ such that $|x|\geq 1,$ 
$$u(x)\leq \max_{|x|=1} u.$$
We then consider $x_0\in\overline{B}(0,1)$ such that $u(x_0)=\max_{\overline{B}(0,1)}u:$ it is a global maximum of $u$ in $\R^2$ at which $Du=0$ and $D^2u\leq 0,$ which is impossible since $\mathcal{H}_1[u]>0.$
\end{proof}

\begin{cor}\label{cor finitness}
Let $c$ be a positive constant. There is no admissible function $u:\R^3\rightarrow\R$ such that 
\begin{equation}\label{ineg u diedre cyl hyp}
|x_1|\leq u(x)\leq\sqrt{x_1^2+c}
\end{equation}
for all $x=(x_1,x_2,x_3)\in\R^3.$ 
\end{cor}

Before the proof, we remark that the function on the RHS of \eqref{ineg u diedre cyl hyp} is the function whose graph is an entire CMC hypersurface $\Sigma$ which is the product of a hyperbola in $\R^{1,1}$ and of $\R^2\cong (\R^{1,1})^\perp$. The LHS is the corresponding function $V_\varphi$, whose graph is the boundary of the domain of dependence of such $\Sigma$, which is obtained as the intersection of the two future half-spaces $x_4>x_1$ and $x_4>-x_1$. Hence $\varphi$ is the function which takes the value $0$ at the points $(\pm 1,0,0)\in\mathbb S^2$, and $+\infty$ elsewhere.

\begin{proof}[Proof of Corollary \ref{cor finitness}]
The section $(x_2,x_3)\mapsto u(0,x_2,x_3)$, with values in $\{x_1=0\}\cong\R^{2,1}$, would have positive mean curvature and would be bounded, which is not possible by Proposition \ref{prop finitness}. Indeed, more generally, a vertical section of a $m$-admissible graph is always $(m-1)$-admissible, see Appendix \ref{appendix Hm}.
\end{proof}

We are now ready to prove Theorem \ref{thm finitness2}.

\begin{proof}[Proof of Theorem \ref{thm finitness2}]
By contradiction, let $u$ be an admissible function satisfying $\mathcal{H}_2[u]\geq h_0>0$ with domain of dependence $\mathcal D_\varphi\subset\R^{3,1}$ such that $card(F_\varphi)=2$. Up to an isometry, we may suppose that (\ref{ineg u diedre cyl hyp}) holds for some coordinates $x_1,x_2,x_3$ and a constant $c$ conveniently chosen. Indeed, the domain of dependence of $u$ is a wedge, that we may assume to be $\{x\in\R^3 \ |\ |x_1|<x_4\}$.  Since $\mathcal{H}_1[u]$ is bounded below by a positive constant $\alpha$ (by the Maclaurin inequality for admissible functions), the graph of $u$ stays below every CMC hypersurface with $\mathcal{H}_1\leq\alpha$ and belonging to the wedge (by the entire comparison principle for the mean curvature operator, Theorem \ref{thm:comparison CMC}). This contradicts Corollary \ref{cor finitness}. 
\end{proof}

\subsection{Higher dimension}
We prove here that the results in the previous section do not hold in higher dimension: there exists a bounded entire spacelike function $u:\R^n\rightarrow\R$ with positive mean curvature if $n\geq 3,$ which is moreover admissible if $n\geq 5,$ and there exists an entire admissible function whose domain of dependence is a wedge if $n\geq 6.$ This shows that admissibility is not an obstruction to the existence of a solution of $\mathcal{H}_2[u]=1$ in $\R^{n,1}$ with $card(F)=2$, if $n\geq 6$ (in contrast to the case $n=3$ obtained in Corollary \ref{cor finitness}):
\begin{prop}\label{prop higher dimensions radially}
The following holds:
\begin{enumerate}
\item If $n\geq 3$ there exists an entire spacelike function $u:\R^n\rightarrow\R$ such that $\mathcal{H}_1[u]>0$ on $\R^n$ and $\sup_{\R^n}u<+\infty.$
\item If $n\geq 5$ there exists an entire spacelike function $u:\R^n\rightarrow\R$ which is admissible and bounded.
\item If $n\geq 6$ there exists an entire admissible function $u:\R^n\rightarrow\R$ such that 
\begin{equation}\label{ineg u diedre cyl hyp 2}
|x_1|\leq u(x)\leq\sqrt{x_1^2+c}
\end{equation}
for all $x=(x_1,x')\in\R^n.$ 
\end{enumerate}
\end{prop}
\begin{proof}
1. Let us consider a radial function $u: x\in\R^n\mapsto v(|x|)\in\R$ where $v:[0,+\infty)\rightarrow\R$ is $C^2,$ spacelike and such that $v'(0)=0,$ and denote the mean curvature of its graph at $x\in\R^n$ by $h(r),$ where $r=|x|.$ Since, from \eqref{expr shape op}, the principal curvatures of its graph are 
$$\lambda_1=\cdots=\lambda_{n-1}=\frac{1}{\sqrt{1-v'^2}}\frac{v'}{r}\hspace{.5cm}\mbox{and}\hspace{.5cm}\lambda_n=\frac{v''}{(1-v'^2)^{\frac{3}{2}}},$$ 
we have
\begin{equation}\label{h radial sec hd}
h(r)=\frac{1}{n\sqrt{1-v'^2}}\left(\frac{v''}{1-v'^2}+(n-1)\frac{v'}{r}\right),
\end{equation}
which reads
$$\left(r^{n-1}\frac{v'}{(1-v'^2)^{\frac{1}{2}}}\right)'=nr^{n-1}h(r)~.$$
Setting 
\begin{equation}\label{H function h sec hd}
H(r):=\frac{n}{r^{n-1}}\int_0^rs^{n-1}h(s)ds
\end{equation}
for $r>0$ and since $v'(0)=0$, an integration between $0$ and $r$ yields 
\begin{equation}\label{v function H sec hd}
\frac{v'}{(1-v'^2)^{\frac{1}{2}}}=H(r)\hspace{.5cm}\mbox{and}\hspace{.5cm}v'=\frac{H(r)}{\sqrt{1+H^2(r)}}~.
\end{equation}
Let us note that, by (\ref{h radial sec hd}), $h(r)\rightarrow_{r\rightarrow 0} v''(0)$ and (\ref{H function h sec hd}) implies that $H(r)\rightarrow_{r\rightarrow 0} 0;$ so (\ref{v function H sec hd}) holds in fact on $[0,+\infty).$ We deduce that, for all $x\in\R^n,$
\begin{equation}\label{u function H sec hd}
u(x)=u(0)+\int_0^{|x|}\frac{H(r)}{\sqrt{1+H^2(r)}}dr~.
\end{equation}
Conversely, if $h:[0,+\infty)\rightarrow (0,+\infty)$ is a positive continuous function so that the function $H$ defined by (\ref{H function h sec hd}) satisfies 
\begin{equation}\label{int H cv sec hd}
\int_0^{+\infty}\frac{H(r)}{\sqrt{1+H^2(r)}}dr<+\infty
\end{equation}
then $u$ defined by (\ref{u function H sec hd}) is a bounded entire spacelike function with mean curvature $h(|x|)$ at the point $(x,u(x))$.

This is achieved for example if $h$ is chosen such that 
\begin{equation}\label{cond h u bounded sec hd}
\int_0^{+\infty}s^{n-1}h(s)ds=+\infty\hspace{.5cm}\mbox{and}\hspace{.5cm}h(s)=_{s\rightarrow+\infty}o(s^{-2-\varepsilon})
\end{equation}
for some $\varepsilon\in (0,1),$ which is possible if $n\geq 3.$ Indeed,  the two conditions in (\ref{cond h u bounded sec hd}) imply that
$$\int_0^rs^{n-1}h(s)ds=_{r\rightarrow+\infty}o\left(r^{n-2-\varepsilon}\right)$$
and therefore that $H(r)=_{r\rightarrow+\infty}o(r^{-1-\varepsilon}),$ so that (\ref{int H cv sec hd}) holds. 

Let us first note that if $n=2$ the second condition in (\ref{cond h u bounded sec hd}) yields $sh(s)=o(1/s^{1+\varepsilon})$ and therefore $\int_0^{+\infty}sh(s)ds<+\infty,$ so the two conditions in (\ref{cond h u bounded sec hd}) are not compatible in that case. This moreover implies that $H(r)\sim C/r$ at $+\infty$ (definition (\ref{H function h sec hd}) with $n=2$) and (\ref{int H cv sec hd}) also fails to hold, as expected in view of the result of the previous section. That discussion shows that the first condition in (\ref{cond h u bounded sec hd}) is in fact necessary to integrate the second condition and obtain (\ref{int H cv sec hd}) if $n\geq 3.$
 
Let us also note that to verify that $u$ is $C^2$ on $\R^n$, especially at $x=0,$ we need to verify that $v':[0,+\infty)\rightarrow\R$ is $C^1$ and such that $v'(0)=0,$ which by  (\ref{v function H sec hd}) is equivalent to prove that $H:[0,+\infty)\rightarrow\R$ is $C^1$ and such that $H(0)=0.$ Since by (\ref{H function h sec hd}) we have $H(r)\sim_{r\rightarrow 0} h(0)r,$ the only critical point is to verify that $H'(r)$ has a limit as $r$ tends to 0: differentiating (\ref{H function h sec hd}) we get
$$(n-1)\frac{H(r)}{r}+H'(r)=nh(r)$$
and deduce that $H'(r)$ tends to $nh(0)$ as $r$ tends to 0, which completes the proof.
\\2. We consider the radial function $u(x)=v(|x|)$ constructed above. The second symmetric function $\sigma_2$ of the principal curvatures of the graph of $u$ is
$$\sigma_2=\frac{n-1}{1-v'^2}\frac{v'}{r}\left(\frac{v''}{1-v'^2}+\frac{n-2}{2}\frac{v'}{r}\right)$$
which also reads using (\ref{h radial sec hd})
\begin{equation}\label{sigma2 function vp sec hd}
\sigma_2=\frac{n-1}{1-v'^2}\frac{v'}{r}\left(nh\sqrt{1-v'^2}-\frac{n}{2}\frac{v'}{r}\right)
\end{equation}
where $h:[0,+\infty)\rightarrow (0,+\infty)$ still denotes the mean curvature. Let us show that if $n\geq 5$ we may choose $h$ satisfying (\ref{cond h u bounded sec hd}) such that this expression is positive. Let us fix $\delta$ such that $0<\delta<n/2-2$ (we assume $n\geq 5$), consider a smooth non-decreasing function $k:[0,+\infty)\rightarrow [0,+\infty)$ such that 
$$k(r)\sim_{r\rightarrow 0} c_0\ r^{n/2}\hspace{.5cm}\mbox{and}\hspace{.5cm}k(r)\sim_{r\rightarrow+\infty}r^\delta$$
for some constant $c_0>0$ and set $h(r)=r^{-n/2}k(r).$ The function $h$ satisfies (\ref{cond h u bounded sec hd}) for every $\varepsilon\in (0,1)$ such that $\delta<n/2-2-\varepsilon.$ Using (\ref{v function H sec hd}), Eq. (\ref{sigma2 function vp sec hd}) may be written
\begin{equation}\label{sigma2 function h sec hd}
\sigma_2=\frac{n(n-1)}{\sqrt{1-v'^2}}\frac{v'}{r^{n+1}}\left(r^nh(r)-\frac{n}{2}\int_0^rs^{n-1}h(s)ds\right).
\end{equation}
The right-hand side term is positive if $r$ is small, since $h(r)$ tends to $c_0$ as $r$ tends to 0. Moreover, the function $r\mapsto r^nh(r)-\frac{n}{2}\int_0^rs^{n-1}h(s)ds$ is non-decreasing since its derivative is
$$\frac{n}{2}r^{n-1}h+r^nh'=r^{n/2}(r^{n/2}h)'=r^{n/2}k'\geq 0.$$
So $\sigma_2$ is positive on $[0,+\infty)$ and $u$ is admissible.
\\3. We construct a solution in the form 
$$u(x)=\sqrt{x_1^2+u'(x')^2},$$
where $x'=(x_2,\ldots,x_n)$ and $u':\R^{n-1}\rightarrow\R$ is admissible and such that $0<u'<\sup_{\R^{n-1}}u'<+\infty,$ constructed above for $n-1\geq 5.$ As a consequence of the construction, (\ref{ineg u diedre cyl hyp 2}) holds. Moreover, it is admissible since $u'$ is admissible and, among the $n$ principal curvatures of $u$, $n-1$ are equal to the principal curvatures of $u'$, while the last principal curvature is positive  (it coincides with the curvature of an hyperbola). Hence, using that $\mathcal H_1[u']>0$ and $\mathcal H_2[u']>0$, one immediately sees that $\mathcal H_1[u]>0$ and $\mathcal H_2[u]>0$.
\end{proof}
\begin{rem} \label{rmk:nonconvexity}
The existence of an admissible solution of $\mathcal{H}_2[u]=1$ in $\R^{n,1}$ with $n\geq 4$ and $card(F_\varphi)=2$ is still an open question. Note that such a solution cannot be convex: in convenient coordinates it would satisfy $|x_1|\leq u(x)\leq\sqrt{x_1^2+c}$ on $\R^n,$ and fixing $x_1\in\R$ the map $x'\in\R^{n-1}\mapsto u(x_1,x')\in\R$ would be convex and bounded, and therefore constant; $u$ would thus only depend on $x_1,$ which is impossible: the graph of an admissible function has necessarily at least two principal curvatures which are positive. This follows from the inequalities
$$\sigma_{1,i}:=\sum_{j\neq i}\lambda_j>0,\ \ i=1,\ldots,n$$
expressing that $\mathcal{H}_2$ is elliptic on admissible functions (see  Section \ref{sec:scalar and admissible}). Note also that it is not known if there exist entire admissible solutions of $\mathcal{H}_2[u]=1$ which are not convex.
\end{rem}

% !TEX root = main.tex

\section{MGHC spacetimes}

The purpose of this section is to prove Theorem \ref{thm:MGHC} about maximal globally hyperbolic Cauchy compact (in short, MGHC) flat spacetimes. 

\subsection{Flat MGHC spacetimes}

Recall that a Lorentzian manifold $M$ is \emph{globally hyperbolic} if it contains a Cauchy hypersurface, namely a smooth spacelike hypersurface that intersects every inextensible causal curve exactly once. It is \emph{maximal} if every isometric embedding of $M$ into a globally hyperbolic Lorentzian manifold that sends a Cauchy hypersurface to a Cauchy hypersurface is surjective. It is \emph{Cauchy compact} if it has a compact Cauchy hypersurface. (It turns out that all Cauchy hypersurfaces are diffeomorphic, and that $M$ is diffeomorphic to the product of a Cauchy hypersurface and $\R$.)

Let us now outline the classification result that was proved in \cite{Ba05}: every maximally globally hyperbolic Cauchy compact flat spacetimes is, up to taking a finite quotient, a translation spacetime, a Misner spacetime, or a twisted product of a Cauchy hyperbolic spacetime by a Euclidean torus. In the rest of the section, we will describe each of these three situations, and prove that in each case there exists a  foliation by hypersurfaces of constant scalar curvature, where the scalar curvature of the leaves is zero for translation spacetimes and Misner spacetimes, whereas it varies monotonically in $(-\infty,0)$ for twisted products of Cauchy hyperbolic spacetimes by Euclidean tori --- the latter being the most interesting case.

\subsection{Translation spacetimes}\label{sec:translation}

A translation spacetime is a quotient of $\R^{n,1}$ by a lattice $\Lambda\cong\Z^n$ contained in $\R^n=\{x_{n+1}=0\}$, hence consisting entirely of spacelike vectors. We start by treating this simple case, which is instructive for the case of Misner spacetimes and, most importantly, twisted products of Cauchy hyperbolic spacetimes by Euclidean tori.

\begin{prop}
Let $M$ be a MGHC spacetime which is finitely covered by a translation spacetime. Then $M$ has a foliation by flat hypersurfaces. Moreover, if $\Sigma_0$ is a closed spacelike hypersurface of constant scalar curvature $S$, then $S=0$.
\end{prop}
\begin{proof}
Every MGH finite quotient $M$ of a translation spacetime is isometric to a quotient $\R^{n,1}/\hat\Lambda$ by a discrete group $\hat\Lambda$  of Euclidean isometries (a crystallographic group) acting freely on $\R^n$. Now, each leaf of the foliation $\{x_{n+1}=c\}$ of $\R^{n,1}$ by horizontal hypersurfaces  is preserved by any subgroup of $\mathrm{Isom}(\R^n)<\mathrm{Isom}(\R^{n,1})$, and therefore induces a foliation by flat Cauchy hypersurfaces of the quotient manifold $M$. 

We now prove the second statement. Let $F:M\to\R$ be a function whose level sets are precisely the leaves of the above foliation and which is increasing in the future direction --- for instance, $F$ is the function induced in the quotient by the $x_{n+1}$ coordinate function. Since $\Sigma_0$ is closed, $F|_{\Sigma_0}$ has a maximum $p_{\max}$ and a minimum $p_{\min}$. 
%Now, let  $\widetilde\Sigma_0$ be the lift of $\Sigma_0$, which is a complete connected hypersurface in $\R^{n,1}$ of constant scalar curvature. We want to show that $\widetilde\Sigma_0$ is a leaf of the foliation of $\R^{n,1}$ by horizontal totally geodesic hyperplanes. 
In the following, we always compute the second fundamental form with respect to the future-directed unit normal vector. All computations that follow are local, so to simplify the notation, we implicitly work in a neighbourhood of $p_{\min}$ or $p_{\max}$ which is identified, via a chart, to an open subset of Minkowski space. 

First, $ \Sigma_0$ is tangent to the totally geodesic hyperplane $F^{-1}(F(p_{\min}))$ and contained in its future. Hence its principal curvatures (with respect to the future unit normal vector) at  $p_{\min}$ are all non-negative, and therefore the $\mathcal H_2$ of $\Sigma_0$, which is constant by hypothesis, is non-negative. It remains to show that the $\mathcal H_2$ of $\Sigma_0$ cannot be positive. But if the $\mathcal H_2$ were positive at every point, then the mean curvature would never vanish by \eqref{eq:silly inequality}, it would be positive at $ p_{\min}$ by the previous observation that all principal curvatures at $ p_{\min}$ are non-negative, and negative at $ p_{\max}$ by the same argument, thus contradicting the continuity of the mean curvature. 
\end{proof}

\begin{rem}\label{rmk:translation not unique} 
The foliation of a translation spacetime $M=\R^{n+1}/\Lambda$ by flat (hence   of vanishing scalar curvature) hypersurfaces is highly non-unique. For example, supposing that $\Lambda$ is the standard lattice $\Z^n$ in $\R^n$ only to simplify the formula, for any $f:\R\to\R$ such that $|Df|<1$ and $f(t+1)=f(t)$, the graphs of the functions $F_c(x_1,\ldots,x_{n+1})=f(x_1)+c$ (for $c\in\R$) are flat hypersurfaces foliating $M$.
\end{rem}

\subsection{Misner spacetimes}\label{sec:Misner}

A Misner spacetime is a quotient of a wedge 
$$W:=\{x\in\R^{n,1}\,|\,x_{n+1}>|x_1|\}~.$$ To exploit this, observe that the metric of $W$ can be written as 
$$t^2ds^2+dx_2^2+\ldots+dx_{n}^2-dt^2$$
where we have performed a simple change of coordinates from $(x_1,x_{n+1})$ to $(t,s)$ in the copy of $\R^{1,1}$ given by $x_2=\ldots=x_{n}=0$, $t=\sqrt{x_{n+1}^2-x_1^2}\in(0,+\infty)$ is the timelike distance from the origin and $s\in\R$ is the arclength parameter of every hyperbola of the form $x_{n+1}^2-x_1^2=t^2$. From this expression, one sees that every Minkowski isometry preserving $W$ acts by Euclidean isometries on the hypersurface $\{t=1\}$, which is intrinsically isometric to $\R^n$, and by the identity on the $t$ factor. The quotient of $W$ by a discrete subgroup $\Lambda$ acting on $\{t=1\}$ as a lattice is called a Misner spacetime.

\begin{prop}
Let $M$ be a MGHC spacetime which is finitely covered by a Misner spacetime. Then $M$ has a foliation by flat hypersurfaces. Moreover, if $\Sigma_0$ is a closed spacelike hypersurface of constant scalar curvature $S$, then $S=0$.
\end{prop}
\begin{proof}
We argue as in the case of translation spacetimes. Since every group of isometries of $W$ preserve the $t$-coordinate, it preserves its level sets, which are flat hypersurfaces. Hence the quotient $M$ inherits a foliation by flat hypersurfaces. 

For the second part, as before, let $F:M\to\R$ be a function whose level sets are precisely the leaves of the above foliation  and which is increasing in the future direction. For instance,  $F$ is the function induced by the $t$ coordinate. As in the previous proof, we perform all arguments, which are local, in a small Minkowski chart.

Let $p_{\max}$ and $p_{\min}$ the maximum and minimum of $F|_{\Sigma_0}$. Then $\Sigma_0$ is a hypersurface of constant scalar curvature, which is tangent to the leaf $ F^{-1}(F(p_{\min}))$, and contained in its future. Since every leaf $\{F=c\}$ has one positive principal curvature and $n-1$ zero principal curvatures, by Weyl's monotonicity theorem the principal curvatures of $ \Sigma_0$ at $p_{\min}$ are all non-negative, and at least one is positive. Hence the (constant) $\mathcal H_2$ of $ \Sigma_0$ is non-negative, and moreover the mean curvature is positive at $p_{\min}$.  (Alternatively, we could have used Remark \ref{rmk:geometric max principle tangent} to infer that $\mathcal H_2$ is non-negative, except that the leaf $\{F=c\}$ is the graph of a non-admissible  function --- say, $u$ --- since its $\mathcal H_2$ vanishes, so one has to first perturb $u$ and $v$ by adding a function of the form $\epsilon \|x\|^2$ and then let $\epsilon\to 0$.)

It only remains to exclude the possibility that the $\mathcal H_2$ is positive.
 By contradiction, suppose that the $\mathcal H_2$  of $ \Sigma_0$ is positive. Then the mean curvature of $ \Sigma_0$ never vanishes by \eqref{eq:silly inequality}, and is positive everywhere because it is positive at $p_{\min}$ by the previous discussion on the principal curvatures. Thus $ \Sigma_0$ is the graph of an admissible function and we can apply Remark \ref{rmk:geometric max principle tangent} at $ p_{\max}$. This implies that the $\mathcal H_2$ of $ F^{-1}(F(p_{\max}))$ is positive at $ p_{\max}$, which gives a contradiction.
\end{proof}

\begin{rem}\label{rmk:Misner not unique} 
As in Remark \ref{rmk:translation not unique}, one easily sees that uniqueness of the hypersurfaces of constant scalar curvature does not hold in Misner spacetimes. Taking the standard $\Z^n$ lattice acting on $\{t=1\}$ for simplicity, one can see that the hypersurfaces $\{t=F(x_2,\ldots,x_{n},s)\}$, where $F(x_2,\ldots,x_{n},s)=f(s)$ and $f$ is such that $f(s+1)=f(s)$ and defines a spacelike curve in $\R^{1,1}$, all have vanishing sectional curvature.
\end{rem}

\subsection{The interesting case}

We now turn our attention to the most important part of Theorem \ref{thm:MGHC}, namely the existence of foliations by hypersurfaces of constant scalar curvature in MGHC flat spacetimes which are not (up to finite cover) translation spacetimes or Misner spacetimes. This will follow from Corollary \ref{cor foliation general}. Let us first recall the definitions. 

A \emph{Cauchy hyperbolic} spacetime is obtained as a quotient of a regular domain $\mathcal D_\varphi\subset\R^{n,1}$ by a group $\Gamma$ of isometries of $\R^{n,1}$, acting freely and properly discontinuously on $\mathcal D_\varphi$. It turns out that the linear part of $\Gamma$, namely the projection to $\mathrm{O}(n,1)$, is contained in the identity component of $\mathrm{O}(n,1)$ and acts freely and properly discontinuously on $\mathbb H^n=\{x_1^2+\ldots+x_n^2-x_{n+1}^2=-1,\,x_{n+1}>0\}$, with quotient a closed hyperbolic manifold diffeomorphic to any Cauchy hypersurface of $M$.

Moreover, very importantly, $\varphi:\mathbb S^{n-1}\to\R$ is continuous. This fact can be deduced from various references: it is a consequence of \cite[Theorem 3.1]{Li95} together with the existence of a uniformly convex Cauchy hypersurface (\cite[Theorem 1.6]{bonfill}); it is contained in \cite[Section 4]{barfill}; it is proved in parts (1) and (2) of \cite[Theorem F]{NS23}, specialized to a closed manifold, i.e. the divisible case (the proof is done for $n=2$ but extends immediately to any dimension). 

Now, a \emph{twisted product} of $M$ as above and a Euclidean torus is the quotient of $\mathcal D_\varphi\times T$ (endowed with the product metric), where $T$ is a Euclidean torus, by the action of $\Gamma$ defined by $\gamma(p,q)=(\gamma\cdot p,\rho(\gamma)\cdot q)$, for a representation $\rho:\Gamma\to \mathrm{Isom}(T)$.  When $\rho$ is the trivial representation, $M$ is simply the (untwisted) product of $\mathcal D_\varphi/\Gamma$ and $T$. Observe that we allow $T$ to have dimension $0$ (i.e. $T$ is a point) so as to include Cauchy hyperbolic spacetimes in this definition.

\begin{cor}
Let $M$ be a MGHC spacetime which is finitely covered by a twisted product of a Cauchy hyperbolic spacetime with a Euclidean torus. Then $M$ has a foliation by hypersurfaces of constant scalar curvature. Moreover, every spacelike hypersurface of constant scalar curvature in $M$ is a leaf of this foliation.
Finally, the function sending $p$ to the value of the scalar curvature of the unique leaf through $p$ defines a time function on $M$. \end{cor}
\begin{proof}
By the above description, $M$ is isometric to a quotient of $\mathcal D_{\widehat\varphi}$ by the action of a group $\widehat\Gamma$, where $\widehat\varphi:\mathbb S^{n-1}\to\R\cup\{+\infty\}$ is a lower semi-continuous function which is continuous and real-valued when restricted to $\mathbb S^{n-d-1}\subset \mathbb S^{n-1}$ (here $0\leq d=\dim(T)\leq n-2$ and $\mathbb S^{n-d-1}$ is defined by the vanishing of the last $d$ coordinates) and is identically $+\infty$ on the complement of  $\mathbb S^{n-d-1}$. To understand the action of $\widehat\Gamma$, first observe that, since $\mathcal D_{\widehat\varphi}$ splits as $\mathcal D_\varphi\times\R^d$, $\widehat\Gamma$ is a subgroup of the product $\mathrm{Isom}(\R^{n-d,1})\times\mathrm{Isom}(\R^d)$. We have a short exact sequence
$$1\rightarrow \Z^d\rightarrow \widehat \Gamma\rightarrow \Gamma\rightarrow 1$$
where $\Z^d$ is acting as a lattice on $\R^d$, with quotient isometric to $T$, whereas $\Gamma\cong\widehat\Gamma/\Z^d$ is acting on $\mathcal D_\varphi\times T\cong \mathcal D_{\widehat\varphi}/\Z^d$ by the $\rho$-twisted action described above. As a consequence of this discussion, the projection of $\widehat \Gamma$ to $\mathrm{Isom}(\R^{n-d,1})$ coincides with $\Gamma$ seen as acting on $\R^{n-d,1}$ (preserving $\mathcal D_\varphi$).

By Corollary \ref{cor foliation general}, $\mathcal D_{\widehat\varphi}$ admits a foliation by constant scalar curvature hypersurfaces $\widehat\Sigma_S$, where the scalar curvature $S$ varies in $(-\infty,0)$. These are obtained as products of $\R^d$ and of the hypersurfaces $\Sigma_S$ of constant scalar curvature whose domain of dependence is the regular domain $\mathcal D_{\varphi}\subset\R^{n-d,1}$, where $\varphi=\widehat\varphi|_{\mathbb S^{n-d-1}}$. We have shown that the projection of $\widehat\Gamma$ preserves $\mathcal D_\varphi$, and therefore preserves each $\Sigma_S$ by the uniqueness part of Theorem \ref{thm existence geometric}. Therefore $\widehat \Gamma$ preserves each $\widehat\Sigma_S$. Since the action of $\widehat \Gamma$ on $\mathcal D_{\widehat\varphi}$ is free and properly discontinuous, so is the action on $\widehat\Sigma_S$, hence each $\widehat \Sigma_S$ induces a Cauchy hypersurface in the quotient manifold $M$.
This shows the first part of the statement.

Now, the function sending $p$ to the value of the scalar curvature of the leaf through $p$ in $\mathcal D_{\widehat\varphi}$ is increasing along future-directed timelike curves and thus induces a  time function in $M$. This proves the last part of the statement.

It thus remains to show the ``moreover'' part. Consider, as in Sections \ref{sec:translation} and \ref{sec:Misner}, a function $F:M\to\R$ whose level sets are the hypersurfaces of the foliation. Concretely, we take for $F$ the time function described in the previous paragraph, which takes values in $(-\infty,0)$. Given any closed spacelike hypersurface $\Sigma_0$ in $M$ of constant scalar curvature, let $p_{\min}$ and $p_{\max}$ the minimum and maximum points of $F|_{\Sigma_0}$, and let $c$ the (constant) value of the $\mathcal H_2$ of $\Sigma_0$. At $ p_{\min}$, $\Sigma_0$ is tangent to the leaf $F^{-1}(F(p_{\min}))$ and contained in its future. Since the latter has $\mathcal H_2$ identically equal to $n(n-1)|F(p_{\min})|$ and positive mean curvature, we can apply Remark \ref{rmk:geometric max principle tangent} in a small Minkowski chart and infer that 
\begin{equation}\label{ineq c H2 1}
c\geq n(n-1)|F(p_{\min})|~.
\end{equation} Now, by \eqref{eq:silly inequality}, the mean curvature of $\Sigma_0$ never vanishes. But, by Remark  \ref{rmk:geometric max principle tangent} again, the mean curvature of $\Sigma_0$ at $p_{\min}$ is  greater than or equal to that of $F^{-1}(F(p_{\min}))$ at $p_{\min}$, which is positive, so $\Sigma_0$ has positive mean curvature everywhere. In other words, $\Sigma_0$ is locally, in any Minkowski chart, the graph of an admissible function. Then we can apply Remark \ref{rmk:geometric max principle tangent} at $p_{\max}$ and conclude that 
\begin{equation}\label{ineq c H2 2}
c\leq n(n-1)|F(p_{\max})|~.
\end{equation}
 Putting \eqref{ineq c H2 1} and \eqref{ineq c H2 2} together, we have $|F(p_{\min})|\leq |F(p_{\max})|$. But $F(p_{\min})\leq F(p_{\max})<0$, therefore $|F(p_{\min})|\geq |F(p_{\max})|$, so we have shown that $F(p_{\max})= F(p_{\min})$, i.e. $F|_{\Sigma_0}$ is constant. This concludes the proof.
\end{proof}

% !TEX root = main.tex

\appendix

\section{Some algebraic properties of the curvature operators}\label{appendix Hm}
Let us denote by $S_n(\R)$ the set of $n\times n$ symmetric matrices with real coefficients, and set, for $p\in B(0,1)\subset\R^n$ and $q\in S_n(\R),$
\begin{equation}\label{expr shape op pq}
h^i_j=\frac{1}{\sqrt{1-|p|^2}}\sum_{k=1}^n\left(\delta_{ik}+\frac{p_i\ p_k}{1-|p|^2}\right)q_{kj}.
\end{equation}
Let us denote by $\mathcal{H}_k(p,q)$ the normalised $k^{th}$ elementary symmetric function of the eigenvalues $\lambda_1,\ldots,\lambda_n$ of $(h^i_j)_{i,j}$
$$\mathcal{H}_k(p,q)=\frac{k!(n-k)!}{n!}\ \sigma_k(\lambda_1,\ldots,\lambda_n),$$ 
which reads
\begin{equation}\label{app Hk p q}
\mathcal{H}_k(p,q)=\frac{k!(n-k)!}{n!}\frac{1}{\left(1-|p|^2\right)^{\frac{k}{2}}}\sum_{I,J}\left(\delta_{ij}+\frac{p_ip_j}{1-|p|^2}\right)_{I,J}q_{I,J}
\end{equation}
where the sum is over all the multi-indices $I=i_1<\cdots<i_k,$ $J=j_1<\cdots<j_k,$ and where, if $A$ is a $n\times n$ matrix, $A_{I,J}$ stands for the determinant of the $k\times k$ matrix formed by the lines of indices $I$ and the columns of indices $J$ of $A.$  
Let us set, for $p\in B(0,1)\subset\R^n,$ the positive cone associated to the operator $\mathcal{H}_m,$
\begin{eqnarray*}
\Gamma_m(p)&=&\{q\in S_n(\R)|\ \forall \eta\in S_n(\R)^{+*}\ \mathcal{H}_m(p,q+\eta)>\mathcal{H}_m(p,q)>0\}\\
&=&\{q\in S_n(\R)|\ \mathcal{H}_k(p,q)>0,\ k=1,\ldots,m\}
\end{eqnarray*}
where $S_n(\R)^{+*}$ stands for the set of $n\times n$ symmetric matrices which are positive definite. Let us finally define the set of positivity of $\mathcal{H}_m$ by
$$\mathcal{E}:=\{(p,q)\in B(0,1)\times S_n(\R)|\ q\in\Gamma_m(p)\}$$
and recall that the following properties hold on $\mathcal{E}$:
\begin{itemize}
\item the operator $\mathcal{H}_m$ is elliptic: for all $(p,q)\in\mathcal{E}$ and all $\xi\in\R^n\backslash\{0\},$
\begin{equation}\label{Hm elliptic}
\sum_{i,j}\frac{\partial\mathcal{H}_m}{\partial q_{ij}}(p,q)\ \xi_i\ \xi_j>0\ ;
\end{equation}
\item the operator $\mathcal{H}_m^{\frac{1}{m}}$ is concave with respect to the second variable $q$: for all $(p,q)\in\mathcal{E}$ and all $(\xi_{ij})_{ij}\in S_n(\R)$, 
\begin{equation}\label{Hm concave}
\sum_{i,j,\ k,l}\frac{\partial^2\mathcal{H}_m^{\frac{1}{m}}}{\partial q_{ij}\partial q_{kl}}(p,q)\ \xi_{ij}\ \xi_{kl}\leq 0\ ;
\end{equation}
\item the Maclaurin inequalities hold: for all $(p,q)\in\mathcal{E}$ and all $i\in\{1,\ldots,m-1\},$
\begin{equation}
\left({\mathcal{H}_i(p,q)}\right)^{\frac{1}{i}}\geq \left({\mathcal{H}_{i+1}(p,q)}\right)^{\frac{1}{i+1}}\ .
\end{equation}
\end{itemize}
The following result shows that uniform ellipticity is granted once $p$ and $q$ are bounded:
\begin{prop}\label{prop ell unif appendix Hm}
Let $\delta>0,$ $\theta\in (0,1]$ and $C\geq 0$ be given constants. There exist two positive constants $\lambda$ and $\Lambda$ depending only on $\delta,$ $\theta$ and $C$ such that, for all $(p,q)\in\mathcal{E}$ satisfying $|p|\leq 1-\theta,$ $|q|\leq C$ and $\mathcal{H}_m(p,q)\geq\delta$,
$$\lambda|\xi|^2\leq\sum_{k,l}\frac{\partial \mathcal{H}_m}{\partial q_{kl}}(p,q)\xi_k\xi_l\leq\Lambda|\xi|^2$$
for all $\xi\in\R^n.$ 
\end{prop}
\begin{proof}
The upper bound is straightforward and the lower bound relies on the following inequality: $\forall (p,q)\in\mathcal{E},$
$$\sum_{k,l}\frac{\partial \mathcal{H}_m}{\partial q_{kl}}(p,q)\xi_k\xi_l\geq\frac{1}{n-m+1}(1-|p|^2)\frac{\mathcal{H}_m(p,q)}{|q|_1}|\xi|^2$$
for all $\xi\in\R^n,$ with $|q|_1=\sum_{ij}|q_{ij}|;$ see Lemma 6 in \cite{Iv1} for the key inequality in $\Gamma_m.$
\end{proof}
We finish that section showing that vertical sections of admissible graphs are admissible. We begin with a useful formula:
\begin{lem}
For all $p\in B(0,1)\subset\R^n$ and all $q\in S_n(\R),$ the formula
$$\frac{\partial\mathcal{H}_k}{\partial q_{11}}(p,q)=\frac{n}{k}\frac{\left(1-|p'|^2\right)^{\frac{k+1}{2}}}{\left(1-|p|^2\right)^{\frac{k}{2}+1}}\mathcal{H}_{k-1}(p',q')$$
holds for $k\geq 2,$ where $p'=(p_i)_{2\leq i\leq n}$ and $q'=(q_{ij})_{2\leq i,j\leq n}.$
\end{lem}
\begin{proof}
This is an elementary direct computation using (\ref{app Hk p q}). See \cite{ILT} for a similar formula for the $k^{th}$-curvature operator in euclidean space.
\end{proof}
\begin{cor}
Keeping the notation introduced above, if $q$ belongs to $\Gamma_m(p)$ then $q'$ belongs to $\Gamma_{m-1}(p').$
\end{cor}
\begin{proof}
This is a consequence of the lemma, since $\frac{\partial\mathcal{H}_k}{\partial q_{11}}(p,q)>0$ by the ellipticity of $\mathcal{H}_k$ on $\mathcal{E}=\{(p,q)\in B(0,1)\times S_n(\R)|\ q\in\Gamma_m(p)\},$ for all $k=2,\ldots,m.$
\end{proof}
A spacelike function $u:\R^n\rightarrow\R$ of class $C^2$ is said to be $m$-admissible if $D^2u(x)$ belongs to $\Gamma_m(Du(x))$ for all $x\in\R^n.$ We readily obtain as a corollary that vertical sections of $m$-admissible graphs are $(m-1)$-admissible:
\begin{cor}
If $u:\R^n\rightarrow\R$ is an $m$-admissible function, then $u':\R^{n-1}\rightarrow\R$ defined by $u'(x_2,\ldots,x_n):=u(0,x_2,\ldots,x_n)$ is $(m-1)$-admissible.
\end{cor}
\section{The Dirichlet problem between barriers}\label{section Dirichlet}

We solve here the Dirichlet problem for the prescribed scalar curvature equation between two barriers, with a boundary condition given by the upper barrier:
\begin{thm}\label{thm DP barriers}
Let $\Omega$ be a uniformly convex domain in $\R^n$ with $\partial\Omega$ $C^{4,\alpha}$ for some $\alpha\in (0,1),$ and $H\in C^{2,\alpha}(\overline{\Omega}\times\R)$ be a positive function. Let $\varphi_1\in C^2(\overline{\Omega})$ be an admissible function and $\varphi_2\in C^{4,\alpha}(\overline{\Omega})$ be strictly convex and spacelike, such that
$$\mathcal{H}_2[\varphi_1]\geq H(.,\varphi_1),\hspace{.5cm}\mathcal{H}_2[\varphi_2]\leq H(.,\varphi_2)\mbox{ in }\Omega$$
and $\varphi_1<\varphi_2$ in $\overline{\Omega}.$ Then there exists a spacelike function $u$ belonging to $C^{4,\alpha}(\overline{\Omega})$ such that
\begin{equation}\label{eq thm DP barriers}
\left\{\begin{array}{rcl}
\mathcal{H}_2[u]&=&H(.,u)\ \mbox{ in }\Omega\\
u&=&\varphi_2\ \mbox{ on }\partial\Omega
\end{array}\right.
\end{equation}
and $\varphi_1\leq u\leq\varphi_2.$ If $\partial_{x_{n+1}}H\geq 0$ the solution is unique.
\end{thm}
\begin{rem}\label{rem appendix varphi1}
Let us note that if $H$ is bounded above and $\varphi_2$ is given, it is immediate to find $\varphi_1$ satisfying the other conditions in the theorem, so that (\ref{eq thm DP barriers}) is solvable: we may take for $\varphi_1$ a function smaller than $\varphi_2$ whose graph is an hyperboloid with scalar curvature $\mathcal{H}_2[\varphi_1]=\sup_{\overline{\Omega}\times\R}H.$
\end{rem}
\begin{proof}
A very similar result was proved in \cite[Theorem 2.1] {Ba1}, with a boundary condition given by the lower barrier $\varphi_1$ instead of the upper barrier $\varphi_2.$ We only point out the slight differences in the proof, and will refer to \cite{Ba1} for the other arguments. We consider the compact set
$$K=\{(x,z):\ x\in\overline{\Omega},\ \varphi_1(x)\leq z\leq\varphi_2(x)\}$$
and the non-negative constant
$$k=\max\left(\sup_K\frac{1}{H}\frac{\partial H}{\partial x_{n+1}},0\right)$$
so that the function $z\mapsto H(x,z)e^{-kz}$ is decreasing on $[\varphi_1(x),\varphi_2(x)]$ for all $x\in\overline{\Omega}.$ We suppose that $\varphi_2$ is not a solution. We consider the Banach space
$$E=\{\overline{v}\in C^{2,\alpha}(\overline{\Omega}):\ \overline{v}=0\mbox{ on }\partial\Omega\},$$
the convex open set of $E$
$$W=\{\overline{v}\in E:\ \overline{v}>0\mbox{ in }\Omega,\ \partial_n\overline{v}>0\mbox{ on }\partial\Omega,\mbox{ and } \overline{v}<\varphi_2-\varphi_1\mbox{ on }\overline{\Omega}\}$$
where $\partial_n$ denotes the interior normal derivative at the boundary and the map
\begin{eqnarray*}
T:\hspace{.5cm} [0,1]\times W&\rightarrow& E\\
(t,\overline{v})&\mapsto&\overline{u}
\end{eqnarray*}
where $\overline{u}\in E$ is such that $u=\varphi_2-\overline{u}$ is the admissible solution of the Dirichlet problem
\begin{equation}
\left\{\begin{array}{rcl}
\mathcal{H}_2[u]e^{-ku}&=&tH(.,v)e^{-kv}+(1-t)H(.,\varphi_2)e^{-k\varphi_2}\ \mbox{ in }\Omega\\
u&=&\varphi_2\ \mbox{ on }\partial\Omega
\end{array}\right.
\end{equation}
(Theorem 1.1 in \cite{Ur}). Here $v=\varphi_2-\overline{v}.$ We may then follow the lines of \cite[Section 2]{Ba1} without modification, and prove that the functions $\varphi_1$ and $\varphi_2$ are still sub- and supersolutions of that Dirichlet problem (using that the function $z\mapsto H(x,z)e^{-kz}$ is decreasing), $T$ takes values in $W$ and the fixed points of $T$ satisfy the estimate
$$\|\overline{u}\|_{2,\alpha}<C$$
for a controlled positive constant $C.$ Considering the convex set $W_c=\{\overline{v}\in W:\ \|\overline{v}\|_{2,\alpha}<C\}$ and $T:[0,1]\times W_c\rightarrow E,$ we then show that $T$ is continuous and compact, $T(0,.)$ maps $\partial W_c$ into $\overline{W_c}$ and $T(t,.)$ does not have any fixed point on $\partial W_c$. The fixed point theorem of Browder-Potter finally implies that $T(1,.)$ has a fixed point, which proves the theorem. The details are carried out in \cite{Ba1}.
\end{proof}

\section{A continuity lemma}

Let $\Sigma$ be an entire spacelike $C^1$ surface in $\R^{n,1}$. Let $\Xi$ be the subset of $\R^{n,1}\times\mathbb S^{n-1}$ consisting of those pairs $(p,v)$ such that the null ray $t\mapsto p+t(v,1)$ intersects $\Sigma$. Observe that, since $\Sigma$ is spacelike, i.e. the graph of a strictly 1-Lipschitz function, the intersection point between the ray $p+t(v,1)$ and $\Sigma$ is unique. 

\begin{rem}\label{rem domain omega appendix}
Note that, by definition of domain of dependence, $\Xi$ includes all pairs $(p,v)$ where $p$ is in $\mathcal D(\Sigma)$. 
\end{rem}

\begin{lem}\label{lemme:continuité}
The subset $\Xi\subset \R^{n,1}\times\mathbb S^{n-1}$ is open, and the map associating to $(p,v)$ the unique point of intersection between the ray $t\mapsto p+tv$ and $\Sigma$ is continuous.
\end{lem}
\begin{proof}
We prove the statement by the implicit function theorem. Let $u:\R^n\to\R$ be the spacelike function whose graph is $\Sigma$. Finding the intersection point between the ray $p+t(v,1)$ and $\Sigma$ amounts to finding $t\in\R$ such that $(x+tv,y+t)$ is in the graph of $u$, where $p=(x,y)$ for $x\in\R^n$ and $y\in\R$. That is, $(x,y,v,t)$ is a solution of the equation $F(x,y,v,t)=0$ where
$$F(x,y,v,t):=y+t-u(x+tv)=0~.$$
Now, suppose $(x_0,y_0,v_0,t_0)$ is a solution. 
We have
$$\partial_tF(x_0,y_0,v_0,t_0)=1-\langle Du({x_0+t_0v_0}), v_0\rangle~.$$
Since $u$ is spacelike, $|Du|<1$. Together with $|v_0|=1$, we thus have $\partial_tF(x_0,y_0,v_0,t_0)\neq 0$. This shows that, as $(x,y,v)$ vary in a small neighbourhood of $(x_0,y_0,v_0)$ in $\R^n\times\R\times\mathbb S^{n-1}$, all solutions can be expressed as $(x,y,v,t(x,y,v))$ where $t=t(x,y,v)$ is a continuous function. This concludes the proof. 
\end{proof}

\bibliographystyle{alpha}
%\bibliographystyle{ieeetr}
%\addtocontents{toc}{\SkipTocEntry}
\bibliography{biblio_scalar.bib}

\end{document}